%% file: main.tex
\documentclass[letterpaper, oneside]{amsart}

\usepackage[english]{babel}
\usepackage[utf8]{inputenc}
\usepackage[T1]{fontenc}
\usepackage{fancyhdr}
\usepackage{csquotes}
\usepackage{cite}

\usepackage{amssymb}
\usepackage{amsmath}
\usepackage{nicefrac}
\usepackage{bbm}
\usepackage{amsfonts}
\usepackage{tensor}
\usepackage{amsthm}
\usepackage{mathrsfs} 
\usepackage{mathtools}
\usepackage{extpfeil}
\usepackage{tikz-cd}
\usepackage{bm}
\usepackage{bbold}
\usepackage[mathscr]{euscript}

\usepackage{color}
\usepackage{graphicx}
\usepackage{wrapfig}
\usepackage{subfigure}
\usepackage{placeins}

\usepackage{multicol}
\usepackage{ltxtable}
\usepackage{url}
\usepackage{geometry}
\usepackage{paralist}
\usepackage{verbatim}
\usepackage{hyperref}
\usepackage{glossaries}
\usepackage{float}
\usepackage{nameref}
\usepackage{array}
\usepackage{marginnote}
\usepackage{adjustbox}
\usepackage{blindtext}
\usepackage{todonotes}
\usetikzlibrary{decorations.pathmorphing} 
\usepackage{quiver}

\newcommand{\field}{\mathbb{k}}

\DeclareMathOperator{\map}{Map}

\DeclareMathOperator{\alg}{Alg}
\DeclareMathOperator{\e}{End}
\DeclareMathOperator{\lmod}{LMod}

\DeclareMathOperator{\ch}{Ch}

\DeclareMathOperator{\sym}{\textnormal{Sym}}
\DeclareMathOperator{\an}{\mathscr{S}}

\DeclareMathOperator{\defm}{\textnormal{Def}}

\DeclareMathOperator{\modu}{\textnormal{Mod}}
\DeclareMathOperator{\fmp}{\textnormal{\textbf{FMP}}}
\DeclareMathOperator{\spc}{\textnormal{Spc}}
\newcommand{\obar}{\textnormal{Bar}}
\newcommand{\cobar}{\textnormal{Cobar}}
\newcommand{\ccoalg}{\textnormal{CoAlg}^{\textnormal{conil}}}
\newcommand{\ccu}{\mathscr{C}^{\textnormal{cu}}}
\newcommand{\coalg}{\textnormal{CoAlg}}
\newcommand{\lie}{\mathbf{Lie}}
\newcommand{\art}{\alg^{\textnormal{Art}}_{\mathscr{P}}}
\newcommand{\faut}{\widehat{\mathcal{A}\textnormal{ut}}}

\newcommand{\nccoassoc}{\mathbf{coAssoc}^{\textnormal{nc}}}

\newcommand{\nuassoc}{\mathbf{Assoc}^{\textnormal{nu}}}

\newcommand{\dgcomm}{\mathbf{Com}}
\newcommand{\nucomm}{\mathbf{Com}^{\textnormal{nu}}}
\newcommand{\dgcocomm}{\mathbf{coCom}}
\newcommand{\nccocomm}{\mathbf{coCom}^{\textnormal{nc}}}

\renewcommand{\hom}{\textnormal{Hom}}

\DeclareMathOperator{\otimesC}{\otimes_{\mathscr{P}^{\text{!`}}}}

\newcounter{intro}

\newtheorem{theorem}{Theorem}
\newtheorem{Itheorem}[intro]{Theorem}
\newtheorem{lemma}[theorem]{Lemma}
\newtheorem{proposition}[theorem]{Proposition}
\newtheorem{corollary}[theorem]{Corollary}

\theoremstyle{definition}
\newtheorem{definition}[theorem]{Definition}

\newtheorem{example}[theorem]{Example}

\theoremstyle{remark}
\newtheorem{remark}[theorem]{Remark}
\newtheorem*{remark*}{Remark}

\newcommand{\displaypar}{%
  \par\addvspace{\abovedisplayskip}%
}
\title{The center is the formal internal automorphism group}
\author[S. Farr]{Sonja M. Farr}
\address{Max Planck Institute for Mathematics, Bonn, Germany}
\email{farr@mpim-bonn.mpg.de}

\begin{document}
\begin{abstract}
We prove that the deformation functor associated to an algebra over a sufficiently nice Koszul operad can be lifted to an operadic formal moduli problem over the Koszul dual operad. We further show that this operadic formal moduli problem corresponds to the derived center of the algebra under the operadic Lurie-Pridham correspondence. As an application, we prove that the classical Gerstenhaber algebra structure on polyvector fields of a commutative algebra agrees with the derived center of the algebra as an abelian Poisson algebra.
\end{abstract} 
\maketitle
\tableofcontents
\section{Introduction} \label{introduction}
\input{introduction}
\section{Background} \label{background}
\input{background}
\section{The formal automorphism group} \label{the_formal_automorphism_group}
\input{the_formal_automorphism_group}
\appendix
\section{Compatibility of Lie additivity and Dunn additivity}\label{appendixA}
\input{appendix}
\bibliographystyle{alpha}
\inputencoding{utf8}    
\bibliography{main}
\end{document}

%% file: introduction.tex
It is well-known that the Hochschild cochain complex of an associative algebra $A$ controls its deformation theory. More precisely, the shifted truncated Hochschild cochain complex $C^{\geq 1}(A,A)[1]$ is a dg Lie algebra which under the Lurie-Pridham correspondence yields the formal neighborhood of $A$ in the moduli space of associative algebras.\displaypar
The Hochschild cochain complex carries additional algebraic structure in the form of a homotopy commutative cup product, making $C^{\ast}(A,A)$ into an $\mathbb{E}_2$-algebra in the $\infty$-category of cochain complexes. One can interpret this additional algebraic structure in two ways. On the one hand, by \cite[Theorem A]{FARR}, one can identify the Hochschild cochain complex with the derived center of $A$, where $A$ is regarded as an $\mathbb{E}_1$-algebra in cochain complexes. The derived center carries a canonical $\mathbb{E}_1$-algebra structure in the $\infty$-category of $\mathbb{E}_1$-algebras in cochain complexes, and thus, by Lurie's Dunn additivity theorem, an $\mathbb{E}_2$-algebra structure. On the other hand, J. Francis showed in \cite{FR} that the full Hochschild cochain complex corresponds to a formal moduli problem that can be lifted to an $\mathbb{E}_2$-formal moduli problem, and therefore obtains a natural $\mathbb{E}_2$-algebra structure. In fact, J. Lurie showed in \cite{DAGX} that these two $\mathbb{E}_2$-algebra structures agree. \displaypar
In this paper, we show that this relationship between the deformation theory of an algebra and its center extends to a large class of algebras over Koszul operads. Let $\mathscr{P}$ be a Koszul operad and $A$ a $\mathscr{P}$-algebra satisfying the connectivity assumptions of Section \ref{the_formal_automorphism_group}. The formal neighborhood of $A$ in the moduli space of $\mathscr{P}$-algebras is a commutative formal moduli problem called the deformation problem of $A$. Its associated formal group is the formal automorphism group of $A$. The Lie algebra of this formal group, which is the dg Lie algebra associated to the deformation problem of $A$ under the Lurie-Pridham correspondence, is called the operadic tangent complex of $A$. \displaypar
To produce an operadic lift of this picture, we use the theory of operadic formal moduli problems developed in \cite{CCN}, and in particular the operadic Lurie-Pridham correspondence \cite[Theorem 1.1]{CCN} for the shifted Koszul dual operad $\mathscr{P}^!\{-1\}$
\begin{align*}
    \Psi_{\mathscr{P}^!\{-1\}}: \alg_{\mathscr{P}\{1\}}(\modu_{\field}) \xrightarrow{\simeq} \fmp_{\mathscr{P}^!\{-1\}}.
\end{align*} 
If the Koszul dual cooperad $\mathscr{P}^{\text{!`}}$ is equipped with a so-called admissible Hopf counital structure, the $\infty$-category of $\mathscr{P}$-algebras, and thus also the $\infty$-category of $\mathscr{P}^!\{-1\}$-formal moduli problems, admits a symmetric monoidal structure. We prove in Proposition \ref{prop1} that in this case, the $\infty$-category of $\mathscr{P}$-algebras admits all centers and centralizers.\displaypar
We define a formal $\mathscr{P}^!\{-1\}$-group to be an $\mathbb{E}_1$-algebra in $\fmp_{\mathscr{P}^!\{-1\}}$. In Section \ref{the_formal_automorphism_group_1}, we construct an enhancement of the formal automorphism group of $A$ as a functor
\begin{align*}
    \underline{\faut}_{\mathscr{P}}(A): \alg^{\text{Art}}_{\mathscr{P}^!\{-1\}} \rightarrow \an
\end{align*}
on $\mathscr{P}^!\{-1\}$-Artin algebras. We call this the formal internal automorphism group of $A$. Our main result is the following.
\begin{Itheorem}[Theorem \ref{thm21}]
The formal internal automorphism group is a formal $\mathscr{P}^!\{-1\}$-group. Under the operadic Lurie-Pridham correspondence, the underlying $\mathscr{P}^!\{-1\}$-formal moduli problem corresponds to the shifted centralizer $\mathfrak{Z}(\textnormal{id}_A)[-1]\in \alg_{\mathscr{P}\{1\}}(\modu_{\field})$ of the identity on $A$. The formal $\mathscr{P}^!\{-1\}$-group structure corresponds to the canonical $\mathbb{E}_1$-algebra structure on the shifted center $\mathfrak{Z}(A)[-1]$.
\end{Itheorem}
In Section \ref{the_formal_automorphism_group_2}, we prove that our formal internal automorphism group lifts the extended formal automorphism group of $A$ obtained as the loops on the curved deformation problem of $A$. In particular, we show that the underlying Lie algebra of the center of $A$ agrees with the extended operadic tangent complex of $A$.\displaypar
Finally, in Section \ref{the_formal_automorphism_group_3}, we apply this theory to the Lie, associative, and shifted Poisson operads. For the shifted Poisson case, we prove that the center of a $\mathbb{P}_n$-algebra agrees with its convolution center as defined by V. Melani and P. Safronov in \cite{MS}. In particular, we prove the following result resolving a conjecture by Safronov \cite[Remark 1.7]{Saf2}.
\begin{Itheorem}[Corollary \ref{cor3}]
Let $A$ be a smooth finitely presented commutative $\field$-algebra. Consider $A$ as a Poisson algebra with zero bracket. Then its Poisson center $\mathfrak{Z}(A) \in \alg_{\mathbb{E}_1}(\alg_{\mathbb{P}_1}(\modu_{\field})) \simeq \alg_{\mathbb{P}_2}(\modu_{\field})$ is equivalent to the $\mathbb{P}_2$-algebra of polyvector fields
\begin{align*}
    T_{\textnormal{poly}}(A) := \sym_A (\textnormal{Der}_{\field}(A,A)[-1])
\end{align*}
equipped with the zero differential and the Schouten bracket.
\end{Itheorem}
\textbf{Conventions.} Throughout, let $\field$ be a field of characteristic 0. All complexes are cochain graded, with differentials of degree $+1$. Constructions naturally written in homological grading are interpreted by reindexing $C^i=C_{-i}$. For a cochain complex $C^\ast$, we denote by $C^\ast[n]$ the cochain complex with $C^k[n]=C^{k+n}$ and differential $d_{C[n]}=(-1)^n d_C$. We denote the 1-category of cochain complexes by $\ch(\field)$, and for a commutative ring $R$, we denote the derived $\infty$-category of $R$-modules by $\modu_R$. We retain the conventional notation $C_\ast^{\mathrm{CE}}$ for Chevalley--Eilenberg chains. We denote by $\map_{\field}$ the cochain complex of morphisms between cochain complexes.\par
The term ``operad'' refers to not necessarily reduced unital symmetric operads. By ``Koszul operad'' we mean a binary quadratic Koszul operad. In particular, Koszul operads are reduced and canonically augmented. For a dg operad $\mathscr{P}$, we denote by $\mathscr{P}\{k\}$ its shift by $k$. In particular, $A$ is a $\mathscr{P}$-algebra if and only if $A[-k]$ is a $\mathscr{P}\{k\}$-algebra. For $n\geq 1$, we denote by $\mathbb{P}_n$ the non-unital Poisson operad with product of degree $0$ and bracket of degree $1-n$.\par
For a Koszul operad $\mathscr{P}$, we equip the category of conilpotent coalgebras over its Koszul dual cooperad $\mathscr{P}^{\text{!`}}$ with the model structure of B.~Vallette \cite{V}. We denote its weak equivalences by $W_{\mathrm{Kos}}$ and call them ``Koszul weak equivalences''. A coalgebra morphism $f$ belongs to $W_{\mathrm{Kos}}$ iff $\cobar_\kappa(f)$ is a quasi-isomorphism, where $\kappa$ is the canonical Koszul twisting morphism. Cofibrations are degreewise injections, so every coalgebra is cofibrant. The Bar--Cobar adjunction associated to $\kappa$ is a Quillen equivalence between this model category and $\mathscr{P}$-algebras with their usual model structure, whose weak equivalences are quasi-isomorphisms \cite[Theorem~2.2]{V}. When the twisting morphism is clear from the context, we write $\mathrm{Bar}$ and $\mathrm{Cobar}$ for $\mathrm{Bar}_{\kappa}$ and $\mathrm{Cobar}_{\kappa}$, respectively.\par
The $\infty$-category of spaces is denoted by $\an$. 
\displaypar
\textbf{Acknowledgments.} This paper is an expanded version of Chapter 5.4 of the author's Ph.D. dissertation \cite{FARR2}. The techniques used here were largely inspired by D. Tamarkin's paper \textit{The Deformation Complex of a $d$-algebra is a $(d+1)$-algebra} \cite{Tam4}.\par
AI was used to help locate references, review and revise earlier drafts, and work out some of the
explicit formulas in Lemma \ref{lem3}. The author takes full responsibility for the content
of this article.\par
The author would like to thank Rok Gregoric and Chris Rogers for helpful remarks, and the Max Planck Institute for Mathematics in Bonn, where this work was carried out, for its hospitality and support.

%% file: background.tex
\subsection{Deformation theory of algebras over Koszul operads}\label{deformation_theory_of_algebras_over_koszul_operads}
Let $\mathscr{P}$ be a dg operad concentrated in non-positive degrees. The functor
\begin{align*}
    \mathcal{M}_{\mathscr{P}}: \alg_{\dgcomm}(\text{Mod}_{\field}^{\text{cn}}) &\rightarrow \an \\
    R &\mapsto \alg_{\mathscr{P}}(\text{Mod}_R)^{\simeq}
\end{align*}
sending a derived affine scheme $R$ to the space of $\mathscr{P}$-algebras in $R$-modules is an fpqc derived $\infty$-stack called the \textbf{moduli space of $\mathscr{P}$-algebras}. Given a $\field$-point $A \in \alg_{\mathscr{P}}(\text{Mod}_{\field})$, one can consider its formal neighborhood in $\mathcal{M}_{\mathscr{P}}$
\begin{align*}
    (\mathcal{M}_{\mathscr{P}})^{\wedge}_A: \alg_{\dgcomm}^{\text{Art}} &\rightarrow \an\\
    R &\mapsto \mathcal{M}_{\mathscr{P}}(R) \times_{\mathcal{M}_{\mathscr{P}}(\field)} \{A\} \simeq \alg_{\mathscr{P}}(\text{Mod}_R)^{\simeq} \times_{\alg_{\mathscr{P}}(\modu_{\field})^{\simeq}} \{A\}.
\end{align*}
This is a formal moduli problem in the sense of \cite{DAGX}. An $R$-point of $(\mathcal{M}_{\mathscr{P}})^{\wedge}_A$ is a so-called $R$-deformation of $A$, i.e. a $\mathscr{P}$-algebra $B$ in $R$-modules together with an equivalence $A \simeq B\otimes_R \field$. We call this formal moduli problem the \textbf{deformation problem of $A$}, and denote it by $\text{Def}^{\mathscr{P}}_A$.\displaypar
By \cite{Hin4}, under the Lurie-Pridham correspondence \cite[Theorem 4.55]{Pr},\cite[Theorem 2.0.2]{DAGX}, this formal moduli problem corresponds to the dg Lie algebra given by the \textbf{operadic tangent complex} of the $\mathscr{P}$-algebra $A$.
\begin{definition}
Let $\mathscr{P}$ be a dg operad. The operadic tangent complex of a $\mathscr{P}$-algebra $A$ is the dg Lie algebra of derived $\mathscr{P}$-derivations of $A$
\begin{align*}
    \mathbb{T}^{\mathscr{P}}_A:= \mathbb{R}\text{Der}^{\mathscr{P}}(A,A).
\end{align*}
\end{definition}
If $\mathscr{P}$ is Koszul, this can be computed by the twisted convolution algebra
\begin{align*}
    \mathbb{T}^{\mathscr{P}}_A \simeq \text{Conv}(\mathscr{P}^{\text{!`}}, \e_A)^{\phi} \cong \left(\map_{\field}(\text{Bar}(A),A), d + [\phi,-]\right),
\end{align*}
with $\phi$ the Maurer-Cartan element in the convolution algebra corresponding to the $\mathscr{P}$-algebra structure on $A$, and $d$ induced by the differentials on $A$ and $\mathscr{P}^{\text{!`}}$. \displaypar
In particular, we have
\begin{align*}
    \mathbb{T}_{(\mathcal{M}_{\mathscr{P}})^{\wedge}_A,A}[-1] \simeq \mathbb{T}^{\mathscr{P}}_A
\end{align*}
as dg Lie algebras.
\begin{example}
Let $\mathscr{P} = \nuassoc$ and $A \in \alg_{\nuassoc}$ a non-unital associative algebra. Then the formal moduli problem of deformations of $A$ as a non-unital associative algebra corresponds to the dg Lie algebra
\begin{align*}
\mathbb{T}^{\nuassoc}_A \simeq \text{Conv}(\nccoassoc\{1\},\e_A)^{\phi} \cong \prod_{n\geq 1} \map_{\field}(A^{\otimes n},A)[1-n]
\end{align*}
with the Hochschild differential. This is just the truncated shifted Hochschild cochain complex of $A$
\begin{align*}
\mathbb{T}^{\nuassoc}_A \simeq C^{\geq 1}(A,A)[1].
\end{align*}
This recovers the fact that Hochschild cohomology governs the deformation theory of associative algebras.
\end{example}
\begin{remark}\label{rmk1}
If $\mathscr{P}$ is a dg operad and $A\in \alg_{\mathscr{P}}(\ch(\field))$, then for any $n$ we have an equivalence of functors
\begin{align*}
    \text{Def}^{\mathscr{P}}_A \simeq \text{Def}^{\mathscr{P}\{-n\}}_{A[n]}.
\end{align*}
Consequently, if $\mathscr{P}$ and $A$ are concentrated in non-positive degrees, then $\text{Def}^{\mathscr{P}\{-n\}}_{A[n]}$ is also a formal moduli problem, which furthermore corresponds to the dg Lie algebra
\begin{align*}
    \mathbb{T}^{\mathscr{P}}_A \simeq \mathbb{T}^{\mathscr{P}\{-n\}}_{A[n]}.
\end{align*}
\end{remark}
\subsection{A reminder on operadic formal moduli problems}
The Lurie-Pridham correspondence establishes an equivalence of $\infty$-categories between formal moduli problems and dg Lie algebras, given by sending a dg Lie algebra to its Maurer-Cartan space functor. This equivalence is an instance of the Koszul duality between $\dgcomm$ and $\lie$. In \cite{DAGX}, Lurie extends this correspondence to an equivalence between $\mathbb{E}_n$-formal moduli problems and $\mathbb{E}_n$-algebras using the Koszul self-duality of $\mathbb{E}_n$. In \cite{CCN}, D. Calaque, R. Campos and J. Nuiten extend this theory to more general Koszul operads. We recall their notion of an operadic formal moduli problem and explain how such problems can be expressed in terms of coalgebras.  \displaypar
Let $\mathscr{P}$ be a reduced dg operad concentrated in non-positive degrees. The full $\infty$-subcategory $\art$ of \textbf{Artin $\mathscr{P}$-algebras} is given by $\mathscr{P}$-algebras $A$ such that 
\begin{itemize}
    \item $H^i(A) = 0$ for $i>0$ and $i\ll 0$,
    \item for all $i$, $H^i(A)$ is finite dimensional, and
    \item each $H^i(A)$ is a nilpotent module over the $H^0(\mathscr{P})$-algebra $H^0(A)$.
\end{itemize}
If $\mathscr{P} = \nucomm$ is the non-unital commutative operad, Artin $\nucomm$-algebras recover precisely the augmentation ideals of usual dg Artin $\field$-algebras.\displaypar
A $\mathscr{P}$-formal moduli problem can then be defined as a functor satisfying the $\mathscr{P}$-analogue of the Schlessinger conditions.
\begin{definition}[Definition 2.11 \cite{CCN}]
A \textbf{$\mathscr{P}$-formal moduli problem} is a functor 
\begin{align*}
    X: \art \rightarrow \an
\end{align*}
such that 
\begin{itemize}
    \item $X(0) \simeq \ast$, where $0$ is the zero $\mathscr{P}$-algebra, and
    \item for any $A\in \art$ and any map $A \rightarrow \field[n]$ with $n\geq 1$, the functor $X$ sends the pullback $A\times_{\field[n]} 0$ to a pullback in $\an$.
\end{itemize}
Let $\fmp_{\mathscr{P}}$ denote the $\infty$-category of $\mathscr{P}$-formal moduli problems.
\end{definition}
Lurie's proof of the Lurie-Pridham correspondence uses the Koszul duality functor for the Lie operad to describe the Maurer-Cartan space functor. For a Koszul operad $\mathscr{P}$ concentrated in non-positive degrees, Calaque-Campos-Nuiten follow an analogous strategy and define a Koszul duality functor
\begin{align*}
    \mathfrak{D}: (\alg_{\mathscr{P}}(\modu_{\field}))^{\text{op}} \xrightarrow{\text{Bar}} (\coalg_{\mathscr{P}^{\text{!`}}}(\ch(\field))[W_{\text{Kos}}^{-1}])^{\text{op}} \xrightarrow{(-)^{\vee}} \alg_{\mathscr{P}^!\{-1\}}(\modu_{\field})
\end{align*}
for $\mathscr{P}$. 
\begin{theorem}[Theorem 1.1 \cite{CCN}]\label{cor1}
Let $\mathscr{P}$ be a Koszul operad concentrated in non-positive degrees. There is an equivalence of $\infty$-categories
\begin{align*}
    \Psi_{\mathscr{P}}: \alg_{\mathscr{P}^!}(\modu_{\field}) &\rightarrow \fmp_{\mathscr{P}}\\
    \mathfrak{g} &\mapsto \map_{\alg_{\mathscr{P}^!}(\modu_{\field})}(\tilde{\mathfrak{D}}(-), \mathfrak{g})
\end{align*}
with $\tilde{\mathfrak{D}} \simeq [-1]\circ \mathfrak{D}$.
\end{theorem}
Similar to how a commutative formal space can be described by its cocommutative coalgebra of distributions, we want to describe $\mathscr{P}$-formal moduli problems in terms of $\mathscr{P}^{\vee}$-coalgebras. If $R\in \art$ is a strict Artin $\mathscr{P}$-algebra \cite[Definition 2.5]{CCN}, i.e. if it is built by finitely many square zero extensions, then its linear dual $R^{\vee}$ is a conilpotent $\mathscr{P}^{\vee}$-coalgebra. If $R\in \art$ is non-strict, the linear dual need not be a $\mathscr{P}^{\vee}$-coalgebra. Instead, define 
\begin{align*}
\delta_{\mathscr{P}} = \text{Bar}_{\kappa^{\vee}}\circ \mathfrak{D}: \art \rightarrow (\ccoalg_{\mathscr{P}^{\vee}}(\ch(\field))[W_{\text{Kos}}^{-1}])^{\text{op}}
\end{align*}
with $\kappa^{\vee}: \mathscr{P}^{\vee} \rightarrow (\mathscr{P}^{\text{!`}})^{\vee} $ the dual of the Koszul twisting morphism for $\mathscr{P}$, which is the Koszul twisting morphism for $\mathscr{P}^!\{-1\}$. This recovers the linear dual on strict Artin $\mathscr{P}$-algebras. Using \cite[Theorem 5.1 (A)]{CCN} and \cite[Proposition 4.13]{CCN}, one sees that for any Artin $\mathscr{P}$-algebra $R$, we get an equivalence
\begin{align}\label{eq1}
    R \simeq \delta_{\mathscr{P}}(R)^{\vee}.
\end{align}
Thus, any conilpotent $\mathscr{P}^{\vee}$-coalgebra $C$ describes a functor $\spc(C)$ on $\mathscr{P}$-Artin algebras given by 
\begin{align*}
    \spc(C): \art &\rightarrow \an\\
    R &\mapsto \map_{\ccoalg_{\mathscr{P}^{\vee}}(\ch(\field))[W_{\text{Kos}}^{-1}]}(\delta_{\mathscr{P}}(R), C)
\end{align*}
We can use this functor to identify $\mathscr{P}$-formal moduli problems with $\mathscr{P}^{\vee}$-coalgebras.
\begin{proposition}\label{prop2}
The functor 
\begin{align*}
    \spc: \ccoalg_{\mathscr{P}^{\vee}}(\ch(\field))[W_{\textnormal{Kos}}^{-1}] \rightarrow \textnormal{Fun}(\art, \an)
\end{align*}
is fully faithful with essential image the $\mathscr{P}$-formal moduli problems.
\end{proposition}
\begin{proof}
Direct computation yields
\begin{align*}
    \spc(C)(R) &= \map_{\ccoalg_{\mathscr{P}^{\vee}}(\ch(\field))[W_{\textnormal{Kos}}^{-1}]}(\delta_{\mathscr{P}}(R),C) \\&\simeq \map_{\alg_{\mathscr{P}^!\{-1\}}(\modu_{\field})}(\mathfrak{D}(R), \text{Cobar}_{\kappa^{\vee}}(C)) \\&\simeq \map_{\alg_{\mathscr{P}^!}(\modu_{\field})}(\tilde{\mathfrak{D}}(R), \text{Cobar}_{\kappa^{\vee}}(C)[-1]) \\&\simeq \Psi_{\mathscr{P}}(\text{Cobar}_{\kappa^{\vee}}(C)[-1])(R).
\end{align*}
Since the Cobar construction and $\Psi_{\mathscr{P}}$ are equivalences, the result follows.
\end{proof}
\begin{remark}
For the dg operad $\dgcomm$ of unital commutative algebras and a dg Artin $\field$-algebra $R\in \alg_{\dgcomm}^{\text{Art}}$, we denote by $\delta_{\dgcomm}(R) := \delta_{\nucomm}(\mathfrak{m}_R)^+$ the coaugmented counital cocommutative coalgebra corresponding to the non-counital cocommutative coalgebra $\delta_{\nucomm}(\mathfrak{m}_R)$, where $\mathfrak{m}_R$ is the augmentation ideal of $R$.
\end{remark}
\subsection{Formal groups and operadic formal groups}\label{formal_groups_and_operadic_formal_groups}
The $\infty$-category of commutative formal moduli problems $\fmp$ is equipped with the Cartesian symmetric monoidal structure. We can hence define a formal group to be a group object in $\fmp^{\times}$. In  \cite[Chapter 7.3]{GR2}, D. Gaitsgory and N. Rozenblyum construct the Lie algebra of such a formal group. They also prove \cite[Chapter 5, Theorem 2.3.2]{GR2} that the loops functor
\begin{align*}
    \Omega: \fmp \rightarrow \text{Grp}(\fmp)
\end{align*}
is an equivalence of $\infty$-categories. By \cite[Appendix B]{FARR2}, the Lie algebra associated to a formal moduli problem $X$ under the Lurie-Pridham correspondence is precisely the Lie algebra of the formal group $\Omega X$.\displaypar
The commutative Lurie-Pridham correspondence is symmetric monoidal when equipping the $\infty$-category of Lie algebras with the direct sum symmetric monoidal structure. For more general operads, we want to consider more general symmetric monoidal structures. 
\begin{definition}[Definition 1.12 \cite{Saf}]\label{def2}
Let $\mathscr{C}$ be a dg cooperad. Let $\mathscr{C}^{\text{cu}}$ be the symmetric sequence that agrees with $\mathscr{C}$ in arities $n>0$, and satisfies $\mathscr{C}^{\text{cu}}(0) = \mathscr{C}(0) \oplus \field$. A \textbf{Hopf counital structure} on $\mathscr{C}$ is a Hopf cooperad structure on the dg cooperad $\mathscr{C}^{\text{cu}}$ such that the projection $\mathscr{C}^{\text{cu}} \rightarrow \mathscr{C}$ is a morphism of cooperads and such that the unit (of the algebra structure with respect to the Hadamard tensor product) on $\mathscr{C}^{\text{cu}}(0)$ is given by the inclusion into the second factor.
\end{definition}
If $\mathscr{C}$ admits a Hopf counital structure, then by \cite[Lemma 1.14]{Saf}, the category $\ccoalg_{\mathscr{C}}(\ch(\field)) \simeq \coalg^{\text{coaug}}_{\mathscr{C}^{\text{cu}}}(\ch(\field))$ inherits the structure of a symmetric monoidal category. We call a Hopf counital structure on a dg cooperad $\mathscr{C}$ \textbf{admissible} if the tensor product on $\coalg^{\text{coaug}}_{\mathscr{C}^{\text{cu}}}(\ch(\field))$ preserves Koszul weak equivalences in both variables.\displaypar 
If $\mathscr{P}$ is a Koszul operad and its Koszul dual cooperad $\mathscr{P}^{\text{!`}}$ is equipped with an admissible Hopf counital structure, the equivalence of $\infty$-categories
\begin{center}
    \begin{tikzcd}
             \alg_{\mathscr{P}}(\ch(\field))[W^{-1}] \arrow[r, shift left=1ex, "\text{Bar}"{name=G}] &\ccoalg_{\mathscr{P}^{\text{!`}}}(\ch(\field))[W_{\text{Kos}}^{-1}]\arrow[l, shift left=.5ex, "\text{Cobar}"{name=F}]
            \arrow[phantom, from=F, to=G, , "\scriptscriptstyle\boldsymbol{\simeq}" rotate=0]
        \end{tikzcd}
\end{center}
makes $\alg_{\mathscr{P}}(\modu_{\field}) \simeq \alg_{\mathscr{P}}(\ch(\field))[W^{-1}]$ into a symmetric monoidal $\infty$-category. We denote by $\otimesC$ the tensor product on $\ccoalg_{\mathscr{P}^{\text{!`}}}(\ch(\field))[W_{\text{Kos}}^{-1}]$.\displaypar
Through the equivalence $\spc: \ccoalg_{\mathscr{P}^{\text{!`}}}(\ch(\field))[W_{\text{Kos}}^{-1}] \xrightarrow{\simeq} \fmp_{\mathscr{P}^!\{-1\}}$, we can also equip $\fmp_{\mathscr{P}^!\{-1\}}$ with a symmetric monoidal structure. By construction, this makes the operadic Lurie-Pridham correspondence
\begin{align*}
    \Psi_{\mathscr{P}^!\{-1\}}: \alg_{\mathscr{P}\{1\}}(\modu_{\field}) \rightarrow \fmp_{\mathscr{P}^!\{-1\}}
\end{align*}
symmetric monoidal, where $\alg_{\mathscr{P}\{1\}}(\modu_{\field})$ is equipped with the shifted symmetric monoidal structure.
\begin{definition}
A \textbf{formal $\mathscr{P}^!\{-1\}$-group} is an $\mathbb{E}_1$-algebra object in $\fmp_{\mathscr{P}^!\{-1\}}$.
\end{definition}
The data of a formal $\mathscr{P}^!\{-1\}$-group structure on a $\mathscr{P}^!\{-1\}$-formal moduli problem $\Psi_{\mathscr{P}^!\{-1\}}(\mathfrak{g})$ is equivalent to the data of an $\mathbb{E}_1$-algebra structure in $\mathscr{P}\{1\}$-algebras on the $\mathscr{P}\{1\}$-algebra $\mathfrak{g}$, which in turn is equivalent to an $\mathbb{E}_1$-algebra structure on the $\mathscr{P}$-algebra $\mathfrak{g}[1]$. 
\subsection{Centers and centralizers}
We give a brief overview of Lurie's theory of centers developed in \cite[Section 5.2]{HA}. In addition, we prove that if $\mathscr{P}^{\text{!`}}$ is equipped with a Hopf counital structure, the $\infty$-category $\alg_{\mathscr{P}}(\modu_{\field})$ admits all centers and centralizers. Let $\mathcal{M}$ be an $\infty$-category left tensored over a monoidal $\infty$-category $\mathcal{C}^{\otimes}$. 
\begin{definition}\label{def1}
Let $M \in \mathcal{M}$, and $L_M: \mathcal{C}_{\mathbb{1}/} \rightarrow \mathcal{M}_{M/}$ be given by tensoring with $\text{id}_M$. A \textbf{centralizer} $\mathfrak{Z}(f)$ of a morphism $f: M \rightarrow N$ in $\mathcal{M}$ is a final object of the \textbf{action $\infty$-category of $f$}
\begin{align*}
    \text{Act}(f) := \mathcal{C}_{\mathbb{1}/} \times_{\mathcal{M}_{M/}} (\mathcal{M}_{M/})_{/f}.
\end{align*}
\end{definition}
Thus, a centralizer of $f$ is a universal object $\mathfrak{Z}(f)$ equipped with an action $\alpha: \mathfrak{Z}(f) \otimes M \rightarrow N$ that reduces to $f$ on the unit. If $f = \text{id}_M: M \rightarrow M$ is an identity morphism, then the composition action
\begin{align*}
    \mathfrak{Z}(\text{id}_M) \otimes \mathfrak{Z}(\text{id}_M) \otimes M \xrightarrow{\text{id}_{\mathfrak{Z}(\text{id}_M)} \otimes \alpha} \mathfrak{Z}(\text{id}_M) \otimes M \xrightarrow{\alpha} M
\end{align*}
equips the centralizer $\mathfrak{Z}(\text{id}_M)$ with a multiplication $\mathfrak{Z}(\text{id}_M) \otimes \mathfrak{Z}(\text{id}_M) \rightarrow \mathfrak{Z}(\text{id}_M)$. One can obtain this additional algebraic structure more directly by defining 
\begin{definition}
A \textbf{center} $\mathfrak{Z}(M)$ of an object $M\in \mathcal{M}$ is a final object of the $\infty$-category
\begin{align*}
    \lmod(\mathcal{M}) \times_{\mathcal{M}} \{M\}.
\end{align*}
\end{definition}
In particular, a center $\mathfrak{Z}(M)$ is an $\mathbb{E}_1$-algebra object in $\mathcal{C}$, together with a universal unital action $\mathfrak{Z}(M) \otimes M \rightarrow M$. The $\mathbb{E}_1$-algebra structure on a center is consistent with the above described algebraic structure on a centralizer of an identity map.
\begin{proposition}[Proposition 5.3.1.8 \cite{HA}]
Let $M\in \mathcal{M}$, and suppose there exists a centralizer $\mathfrak{Z}(\text{id}_M)$. Then there exists a center $\mathfrak{Z}(M)$, and a lift of $M$ to a left module over an $\mathbb{E}_1$-algebra $A$ in $\mathcal{C}$ exhibits $A$ as a center of $M$ if and only if the action map $A \otimes M \rightarrow M$ exhibits $A$ as a centralizer of the identity $\text{id}_M$.
\end{proposition}
Taking $\mathcal{M} = \mathcal{C}$ as left tensored over itself, we may simply speak of centralizers and centers in $\mathcal{C}$.\displaypar
To apply these constructions to $\mathscr{P}$-algebras, we need to equip the $\infty$-category of $\mathscr{P}$-algebras with a monoidal structure. While this is automatic for $\infty$-operads, for a dg operad $\mathscr{P}$, the $\infty$-category of $\mathscr{P}$-algebras in $\modu_{\field}$ does not automatically inherit a monoidal structure. It can however be supplied by a Hopf counital structure on the Koszul dual cooperad $\mathscr{P}^{\text{!`}}$.\displaypar
We prove that in this situation, the $\infty$-category of $\mathscr{P}$-algebras admits all centers and centralizers.
\begin{proposition}\label{prop1}
Let $\mathscr{P}$ be a Koszul operad and let $\mathscr{P}^{\text{!`}}$ be equipped with an admissible Hopf counital structure. Then any morphism  in $\alg_{\mathscr{P}}(\modu_{\field})$ admits a centralizer.
\end{proposition}
To be more precise, we will prove the following. Let $A,B \in \alg_{\mathscr{P}}(\modu_{\field})$, and let $f: A \rightarrow B$ be a morphism of $\mathscr{P}$-algebras. Then the induced morphism $\textnormal{Bar}(f): \textnormal{Bar}(A) \rightarrow \textnormal{Bar}(B)$ admits a centralizer $\mathfrak{Z}(\textnormal{Bar}(f))$ in the symmetric monoidal $\infty$-category $\ccoalg_{\mathscr{P}^{\textnormal{!`}}}(\ch(\field))[W_{\textnormal{Kos}}^{-1}]$. 
\begin{lemma}\label{lem2}
There is an equivalence of $\infty$-categories
\begin{align*}
(\ccoalg_{\mathscr{P}^{\textnormal{!`}}}(\ch(\field))_{\textnormal{Bar}(A)/})[W_{\textnormal{Kos}}^{-1}] \xrightarrow{\simeq} (\ccoalg_{\mathscr{P}^{\textnormal{!`}}}(\ch(\field))[W_{\textnormal{Kos}}^{-1}])_{\textnormal{Bar}(A)/},
\end{align*}
where $\ccoalg_{\mathscr{P}^{\textnormal{!`}}}(\ch(\field))_{\textnormal{Bar}(A)/}$ is equipped with the slice category model structure.
\end{lemma}
\begin{proof}
Every object in the Koszul model structure is cofibrant \cite[Theorem 3.11(2)]{DCH}. We can thus apply \cite[Corollary 7.6.13]{C} to the opposite model category, in which $\textnormal{Bar}(A)$ is fibrant. Taking opposites yields the asserted equivalence.
\end{proof}
\begin{proof}[Proof of Proposition \ref{prop1}]
Let $A\in \alg_{\mathscr{P}}(\ch(\field))$ and consider the functor
\begin{align*}
l_A: \ccoalg_{\mathscr{P}^{\textnormal{!`}}}(\ch(\field))_{0/} \rightarrow \ccoalg_{\mathscr{P}^{\textnormal{!`}}}(\ch(\field))_{\text{Bar}(A)/}\\
(0 \rightarrow S) \mapsto (\text{Bar}(A) \rightarrow S\otimesC \text{Bar}(A)).
\end{align*}
given by tensoring with $\text{id}_{\text{Bar}(A)}$. Colimits in $\ccoalg_{\mathscr{P}^{\textnormal{!`}}}(\ch(\field))$ are created by the forgetful functor to $\ch(\field)$, and similarly colimits in $\ccoalg_{\mathscr{P}^{\textnormal{!`}}}(\ch(\field))_{\text{Bar}(A)/}$ are created by the forgetful functor to $\ch(\field)_{\text{Bar}(A)/}$. The underlying functor on complexes of $l_A$ is given by 
\begin{align*}
\ch(\field)_{0/} &\rightarrow \ch(\field)_{\text{Bar}(A)/}\\
(0 \rightarrow C) &\mapsto (\text{Bar}(A) \rightarrow \text{Bar}(A) \oplus C \oplus \text{Bar}(A)\otimes C).
\end{align*}
This functor preserves colimits, since the direct sum and tensor product of complexes do. In particular, $l_A$ preserves colimits. Since $\ch(\field)$ is locally presentable, and $\ccoalg_{\mathscr{P}^{\textnormal{!`}}}(\ch(\field))$ is given by coalgebras over an accessible comonad, it is again locally presentable. Thus, $l_A$ admits a right adjoint.\par
Since cofibrations in $\ccoalg_{\mathscr{P}^{\textnormal{!`}}}(\ch(\field))_{0/}$ and $\ccoalg_{\mathscr{P}^{\textnormal{!`}}}(\ch(\field))_{\text{Bar}(A)/}$ respectively are given by degreewise injections in the underlying complexes, $l_A$ preserves cofibrations. By the admissibility assumption, it also preserves weak equivalences, and is therefore left Quillen. By the above lemma, we thus get an induced adjunction of $\infty$-categories
\begin{center}
    \begin{tikzcd}
             (\ccoalg_{\mathscr{P}^{\textnormal{!`}}}(\ch(\field))[W_{\text{Kos}}^{-1}])_{0/} \arrow[r, shift left=1ex, "L_A"{name=G}] &(\ccoalg_{\mathscr{P}^{\textnormal{!`}}}(\ch(\field))[W_{\text{Kos}}^{-1}])_{\text{Bar}(A)/}\arrow[l, shift left=.5ex, ""{name=F}]
        \end{tikzcd}.
\end{center}
Since every object of $\ccoalg_{\mathscr{P}^{\textnormal{!`}}}(\ch(\field))$ is cofibrant, the left adjoint $L_A$ is again given by tensoring with $\text{id}_{\text{Bar}(A)}$. By \cite[Proposition 5.2.4.2]{HTT}, for every object $c: \text{Bar}(A) \rightarrow C$ in $(\ccoalg_{\mathscr{P}^{\textnormal{!`}}}(\ch(\field))[W_{\text{Kos}}^{-1}])_{\text{Bar}(A)/}$, the comma $\infty$-category $L_A \downarrow c$ admits a final object. Taking $c=\text{Bar}(f): \text{Bar}(A) \rightarrow \text{Bar}(B)$, we therefore get that
\begin{align*}
    &L_A \downarrow \text{Bar}(f) \\&\simeq (\ccoalg_{\mathscr{P}^{\textnormal{!`}}}(\ch(\field))[W_{\text{Kos}}^{-1}])_{0/} \\ &\qquad {}\times_{(\ccoalg_{\mathscr{P}^{\textnormal{!`}}}(\ch(\field))[W_{\text{Kos}}^{-1}])_{\text{Bar}(A)/}} ((\ccoalg_{\mathscr{P}^{\textnormal{!`}}}(\ch(\field))[W_{\text{Kos}}^{-1}])_{\text{Bar}(A)/})_{/\text{Bar}(f)} \\&\simeq \text{Act}(\text{Bar}(f))
\end{align*}
admits a final object, and hence $\text{Bar}(f)$ admits a centralizer.
\end{proof}

%% file: the_formal_automorphism_group.tex
Let $\mathscr{P}$ be a Koszul operad, let $A\in \alg_{\mathscr{P}}(\ch(\field))$ a $\mathscr{P}$-algebra, and assume that there exists an $n$  such that $\mathscr{P}\{n\}$ and the $\mathscr{P}\{n\}$-algebra $A[-n]$ are both concentrated in non-positive degrees.\displaypar
Recall from Section \ref{deformation_theory_of_algebras_over_koszul_operads} that the deformation problem ${\defm}_A$ is the formal neighborhood $(\mathcal{M}_{\mathscr{P}})^{\wedge}_A$ of $A$ in the moduli space $\mathcal{M}_{\mathscr{P}}$ of $\mathscr{P}$-algebras. As described in \ref{formal_groups_and_operadic_formal_groups}, the associated dg Lie algebra of ${\defm}_A$ is the Lie algebra of the formal group $\Omega (\mathcal{M}_{\mathscr{P}})^{\wedge}_A$. This formal group can be identified with $(\Omega_A \mathcal{M}_{\mathscr{P}})^{\wedge}_A$, the formal completion of the loop space at the constant loop at $A$. The derived $\infty$-stack $\Omega_A \mathcal{M}_{\mathscr{P}}$ can be identified as the moduli space $\mathcal{A}\text{ut}_{\mathscr{P}}(A)$ of $\mathscr{P}$-algebra automorphisms of $A$,
\begin{align*}
    \mathcal{A}\text{ut}_{\mathscr{P}}(A): \alg_{\dgcomm}({\modu_{\field}}^{\text{cn}}) &\rightarrow \mathscr{S}\\
    R &\mapsto \text{Aut}_{\alg_{\mathscr{P}}(\modu_R)}(R\otimes A).
\end{align*}
Consequently, the completion $(\Omega_A \mathcal{M}_{\mathscr{P}})^{\wedge}_A$ sends a dg Artin $\field$-algebra $R$ to the \textbf{formal automorphism group}
\begin{align*}
    \faut_{\mathscr{P}}(A)(R) :=  \text{Aut}_{\alg_{\mathscr{P}}(\modu_R)}(R\otimes A) \times_{\text{Aut}_{\alg_{\mathscr{P}}({\modu_{\field}})}(A)} \{\text{id}_A\}.
\end{align*}
For $R\in \alg^{\text{Art}}_{\dgcomm}$, we have $R\otimes A \simeq \map_{\field}(\delta_{\dgcomm}(R), A)$ by (\ref{eq1}), and therefore, using the Bar-Cobar equivalence,
\begin{align*}
\map_{\alg_{\mathscr{P}}(\modu_R)}(R\otimes A, R\otimes A) &\simeq \map_{\alg_{\mathscr{P}}({\modu_{\field}})}(A, R\otimes A)\\&\simeq \map_{\alg_{\mathscr{P}}(\modu_{\field})}(A, \map_{\field}(\delta_{\dgcomm}(R), A)) \\&\simeq \map_{\ccoalg_{\mathscr{P}^{\text{!`}}}(\ch(\field))[W_{\text{Kos}}^{-1}]}(\delta_{\dgcomm}(R)\otimes \text{Bar}(A),\text{Bar}(A))
\end{align*}
where in the last step we use the theory of twisting morphisms as in \cite[Proposition 11.3.2]{LV}. \displaypar
Thus, $\faut_{\mathscr{P}}(A)$ is the restriction along $\delta_{\dgcomm}$ of the functor
\begin{align*}
    (\text{CoAlg}^{\text{coaug}}_{\dgcocomm}[W_{\text{Kos}}^{-1}])^{\text{op}} &\rightarrow \mathscr{S}\\
    S \mapsto \map_{\ccoalg_{\mathscr{P}^{\text{!`}}}[W_{\text{Kos}}^{-1}]}(S\otimes \text{Bar}(A),\text{Bar}(A))&\times_{\map_{\ccoalg_{\mathscr{P}^{\text{!`}}}[W_{\text{Kos}}^{-1}]}(\text{Bar}(A),\text{Bar}(A))}\{\text{id}_{\obar(A)}\}.
\end{align*}
\subsection{The formal internal automorphism group}\label{the_formal_automorphism_group_1}
For the remainder of this section, assume that $\mathscr{P}^{\text{!`}}$ admits an admissible Hopf counital structure as defined in \ref{def2}. In this situation, one can construct an ``internal'' version of the formal automorphism group
\begin{align*}
    \underline{\faut}_{\mathscr{P}}(A): (\ccoalg_{\mathscr{P}^{\text{!`}}}(\ch(\field))[W_{\text{Kos}}^{-1}])^{\text{op}} &\rightarrow \mathscr{S}\\
    S \mapsto \map_{\coalg_{\mathscr{P}^{\text{!`}}}(\ch(\field))[W_{\text{Kos}}^{-1}]}(S\otimesC \text{Bar}(A),\text{Bar}(A))&\times_{\map_{\coalg_{\mathscr{P}^{\text{!`}}}(\ch(\field))[W_{\text{Kos}}^{-1}]}(\text{Bar}(A),\text{Bar}(A))}\{\text{id}_{\obar(A)}\}.
\end{align*}
The following result is the principal link between the center and the deformation theory of $A$.
\begin{proposition}\label{prop4}
The functor $\underline{\faut}_{\mathscr{P}}(A)$ is representable, with representing object given by the centralizer of $\textnormal{id}_{\textnormal{Bar}(A)}$ in $\ccoalg_{\mathscr{P}^{\textnormal{!`}}}(\ch(\field))[W_{\textnormal{Kos}}^{-1}]$. 
\end{proposition}
\begin{proof}
Note first that the centralizer $\mathfrak{Z}(\text{id}_{\text{Bar}(A)})$ exists by Proposition \ref{prop1}. Let $(u: 0 \rightarrow S)\in (\ccoalg_{\mathscr{P}^{\text{!`}}}(\ch(\field))[W_{\text{Kos}}^{-1}])_{0/}$. By \cite[Lemma 2.13]{FARR} and \cite[Lemma 5.5.5.12]{HTT}, we can identify
\begin{align*}
&\map_{(\ccoalg_{\mathscr{P}^{\text{!`}}}(\ch(\field))[W_{\text{Kos}}^{-1}])_{0/}}(u, \mathfrak{Z}(\text{id}_{\text{Bar}(A)}))\\&\simeq \map_{(\ccoalg_{\mathscr{P}^{\text{!`}}}(\ch(\field))[W_{\text{Kos}}^{-1}])_{\text{Bar}(A)/}}(u \otimesC \text{id}_{\text{Bar}(A)},\text{id}_{\text{Bar}(A)}) \\ &\simeq \map_{\coalg_{\mathscr{P}^{\text{!`}}}(\ch(\field))[W_{\text{Kos}}^{-1}]}(S\otimesC \text{Bar}(A),\text{Bar}(A))\times_{\map_{\coalg_{\mathscr{P}^{\text{!`}}}(\ch(\field))[W_{\text{Kos}}^{-1}]}(\text{Bar}(A),\text{Bar}(A))}\{\text{id}_{\obar(A)}\}.
\end{align*}
\end{proof}
We can restrict $\underline{\faut}_{\mathscr{P}}(A)$ along $\delta_{\mathscr{P}^!\{-1\}}$ to obtain a functor
\begin{align*}
    \underline{\faut}_{\mathscr{P}}(A)|_{\alg^{\text{Art}}_{\mathscr{P}^{\text{!}}\{-1\}}}: \alg^{\text{Art}}_{\mathscr{P}^{\text{!}}\{-1\}} \rightarrow \an.
\end{align*}
\begin{theorem}\label{thm21}
Assume that $\mathscr{P}^!\{-1\}$ is concentrated in non-positive degrees. The functor $\underline{\faut}_{\mathscr{P}}(A)|_{\alg^{\textnormal{Art}}_{\mathscr{P}^{\textnormal{!}}\{-1\}}}$ is a $\mathscr{P}^!\{-1\}$-formal group, which under the operadic Lurie-Pridham correspondence \ref{cor1} corresponds to the $\mathbb{E}_1$-algebra in $\mathscr{P}\{1\}$-algebras given by the shifted center $\mathfrak{Z}(A)[-1]$ of $A$ in the symmetric monoidal $\infty$-category of $\mathscr{P}$-algebras.
\end{theorem}
\begin{proof}
By Proposition \ref{prop4}, we have an equivalence of functors
\begin{align*}
    \underline{\faut}_{\mathscr{P}}(A)|_{\alg^{\text{Art}}_{\mathscr{P}^!\{-1\}}} \simeq \spc(\mathfrak{Z}(\text{id}_{\text{Bar}(A)})).
\end{align*}
By Proposition \ref{prop2}, this shows that $\underline{\faut}_{\mathscr{P}}(A)|_{\alg^{\text{Art}}_{\mathscr{P}^!\{-1\}}}$ is a $\mathscr{P}^!\{-1\}$-formal moduli problem, which under the operadic Lurie-Pridham correspondence corresponds to $\text{Cobar}(\mathfrak{Z}(\text{id}_{\text{Bar}(A)}))[-1]$. The functor $\underline{\faut}_{\mathscr{P}}(A)$ has a lax monoidal structure induced by the composition in the formal internal automorphism group. In particular, for $S, T\in \ccoalg_{\mathscr{P}^{\text{!`}}}(\ch(\field))[W_{\text{Kos}}^{-1}]$, we have a canonical map
\begin{align*}
    \underline{\faut}_{\mathscr{P}}(A)(T) \times \underline{\faut}_{\mathscr{P}}(A)(S) &\rightarrow \underline{\faut}_{\mathscr{P}}(A)(T\otimesC S) 
\end{align*}
sending $f: S \otimesC \text{Bar}(A) \rightarrow \text{Bar}(A)$ and $g: T \otimesC \text{Bar}(A) \rightarrow \text{Bar}(A)$ to the composition
\begin{align*}
    T \otimesC S \otimesC \text{Bar}(A) \xrightarrow{\text{id}_T\otimesC f} T \otimesC \text{Bar}(A) \xrightarrow{g} \text{Bar}(A).
\end{align*}
Since lax monoidal presheaves correspond to $\mathbb{E}_1$-algebras in the Day convolution monoidal structure, and the Yoneda embedding is fully faithful and monoidal, this precisely corresponds to the structure of an $\mathbb{E}_1$-algebra on the representing object $\mathfrak{Z}(\text{id}_{\text{Bar}(A)})$. This is the canonical multiplication on the center $\mathfrak{Z}(\text{Bar}(A))$, since both are characterized by composition of the universal action. The symmetric monoidal Bar-Cobar equivalence identifies $\text{Cobar}(\mathfrak{Z}(\text{Bar}(A)))$ with $\mathfrak{Z}(A)$ as $\mathbb{E}_1$-algebras in $\mathscr{P}$-algebras.
\end{proof} 
\subsection{Comparison to the deformation problem of $A$}\label{the_formal_automorphism_group_2}
The counital cocommutative cooperad $\dgcocomm$ is the monoidal unit with respect to the Hadamard tensor product of cooperads. Consider the unit morphism of the Hopf cooperad $(\mathscr{P}^{\text{!`}})^{\text{cu}}$ 
\begin{align*}
    \varphi^{\text{cu}}: \dgcocomm \rightarrow (\mathscr{P}^{\text{!`}})^{\text{cu}}.
\end{align*}
This induces a morphism of cooperads
\begin{align*}
    \varphi: \nccocomm \rightarrow \mathscr{P}^{\text{!`}}.
\end{align*}
By \cite[Theorem 3.11 and Proposition 3.14]{DCH}, 
\begin{center}
    \begin{tikzcd}
             \ccoalg_{\nccocomm}(\ch(\field)) \arrow[r, shift left=1ex, "\varphi_{\ast}"{name=G}] &\ccoalg_{\mathscr{P}^{\textnormal{!`}}}(\ch(\field))\arrow[l, shift left=.5ex, "\varphi^{\text{ext}}"{name=F}]
        \end{tikzcd}
\end{center}
is a Quillen adjunction with respect to the respective Koszul weak equivalences, and further $\varphi_{\ast}$ preserves those weak equivalences. Since $\varphi^{\text{cu}}$ is a morphism of Hopf cooperads, the functor $\varphi_{\ast}$ is strong symmetric monoidal, and so is the induced functor on $\infty$-categories. In particular, the induced right adjoint on $\infty$-categories is lax symmetric monoidal. \displaypar
We want to compare the $\mathscr{P}^!\{-1\}$-formal group given by the formal internal automorphism group to the formal group given by loops of the commutative formal moduli problem of deformations of $A$. To this end, we need to describe how the morphism $\varphi$ yields a comparison between $\mathscr{P}^!\{-1\}$-formal moduli problems and commutative formal moduli problems.
\begin{lemma}
Let $\mathscr{Q}$ and $\mathscr{R}$ be Koszul operads such that $\mathscr{Q}^!\{-1\}$ and $\mathscr{R}^!\{-1\}$ are concentrated in non-positive degrees, and let $\varphi: \mathscr{Q}^{\textnormal{!`}} \rightarrow \mathscr{R}^{\text{!`}}$ be a morphism of cooperads. The following diagram commutes
\begin{center}
\resizebox{\linewidth}{!}{%
\begin{tikzcd}[
  ampersand replacement=\&,
  column sep={3.5cm,between origins},
  row sep={3.5cm,between origins},
  cells={nodes={inner sep=3pt}},
  every label/.append style={font=\small, inner sep=2.5pt}
]
\&
\alg_{\mathcal Q\{1\}}(\modu_{\field})
\&[5cm]
\&
\textnormal{Fun}\!\left(
  \alg^{\textnormal{Art}}_{\mathcal Q^!\{-1\}},\an
\right)
\\
\alg_{\mathcal R\{1\}}(\modu_{\field})
\&
\&
\textnormal{Fun}\!\left(
  \alg^{\textnormal{Art}}_{\mathcal R^!\{-1\}},\an
\right)
\&
\\[-2cm]
\&
\ccoalg_{\mathcal Q^{\textnormal{!`}}}
  (\ch(\field))[W_{\textnormal{Kos}}^{-1}]
\&
\&
\textnormal{PSh}\!\left(
  \ccoalg_{\mathcal Q^{\textnormal{!`}}}
  (\ch(\field))[W_{\textnormal{Kos}}^{-1}]
\right)
\\
\ccoalg_{\mathcal R^{\textnormal{!`}}}
  (\ch(\field))[W_{\textnormal{Kos}}^{-1}]
\&
\&
\textnormal{PSh}\!\left(
  \ccoalg_{\mathcal R^{\textnormal{!`}}}
  (\ch(\field))[W_{\textnormal{Kos}}^{-1}]
\right)
\arrow[from=3-2, to=1-2,
  "{[-1]\circ\cobar_{\mathcal Q}}"', pos=.55]
\arrow[from=1-2, to=1-4,
  "{\Psi_{\mathcal Q^!\{-1\}}}"]
\arrow[from=3-2, to=3-4, "y", pos=.16]
\arrow[from=3-4, to=1-4,
  "{\delta_{\mathcal Q^!\{-1\}}^*}"']
\arrow[from=3-2, to=1-4, "\spc", pos=.52]
\arrow[from=2-1, to=1-2,
  "{\bigl((\varphi^\vee)^!\bigr)^*}"]
\arrow[from=4-1, to=3-2,
  "\varphi^{\textnormal{ext}}"]
\arrow[from=4-3, to=3-4,
  "{(\varphi_*)^*}"']
\arrow[from=2-3, to=1-4,
  "{\bigl((\varphi^\vee)^*\bigr)^*}"', pos=.25]
\arrow[from=4-1, to=2-1,
  "{[-1]\circ\cobar_{\mathcal R}}"]
\arrow[from=2-1, to=2-3, crossing over,
  "{\Psi_{\mathcal R^!\{-1\}}}"', pos=.82]
\arrow[from=4-1, to=4-3, "y"]
\arrow[from=4-3, to=2-3, crossing over,
  "{\delta_{\mathcal R^!\{-1\}}^*}", pos=.5]
\arrow[from=4-1, to=2-3, crossing over,
  "\spc"', pos=.5]
\end{tikzcd}%
}
\end{center}
If $\mathscr{Q}^{\textnormal{!`}}$ and $\mathscr{R}^{\textnormal{!`}}$ are equipped with admissible Hopf counital structures, and $\varphi$ is the restriction of a morphism of Hopf cooperads $\varphi^{\textnormal{cu}}: (\mathscr{Q}^{\textnormal{!`}})^{\textnormal{cu}} \rightarrow (\mathscr{R}^{\textnormal{!`}})^{\textnormal{cu}}$, then all the arrows in the above diagram preserve $\mathbb{E}_1$-algebra objects if we equip coalgebra presheaves with the Day convolution monoidal structure.
\end{lemma}
\begin{proof}
Commutativity follows from naturality of the Bar-Cobar equivalences,
linear duality, and \cite[Proposition~6.10]{CCN}. Under the Hopf compatibility hypotheses, the remaining monoidal compatibilities follow from monoidality of the Yoneda embedding with respect to Day convolution, lax monoidality of precomposition along strong symmetric monoidal
functors, and transport along the displayed equivalences.
Thus all arrows preserve $\mathbb E_1$-algebra objects.
\end{proof}
Applying this to $\varphi: \nccocomm \rightarrow \mathscr{P}^{\text{!`}}$, we obtain a formal group 
\begin{align*}
    ((\varphi^{\vee})^{\ast})^{\ast} \left(\underline{\faut}_{\mathscr{P}}(A)|_{\alg^{\textnormal{Art}}_{\mathscr{P}^{\textnormal{!}}\{-1\}}}\right) \simeq \left((\varphi_{\ast})^{\ast}\underline{\faut}_{\mathscr{P}}(A)\right)|_{\alg^{\text{Art}}_{\nucomm}}
\end{align*}
\begin{proposition}\label{prop3}
The formal group $\left((\varphi_{\ast})^{\ast} \underline{\faut}_{\mathscr{P}}(A)\right)|_{\alg^{\textnormal{Art}}_{\nucomm}}$ corresponds to the Lie algebra given by the \textbf{extended operadic tangent complex}
\begin{align*}
    \mathbb{T}^{\mathscr{P}, \textnormal{ex}}_A := \textnormal{Conv}((\mathscr{P}^{\textnormal{!`}})^{\textnormal{cu}}, \e_A)^{\phi}
\end{align*}
with $\phi$ the Maurer-Cartan element corresponding to the $\mathscr{P}$-algebra structure on $A$.
\end{proposition}
\begin{proof}
The functor $(\varphi_{\ast})^{\ast} \underline{\faut}_{\mathscr{P}}(A)$ sends $S\in \ccoalg_{\nccocomm}[W_{\textnormal{Kos}}^{-1}]$ to 
\begin{align*}
    \map_{\ccoalg_{\mathscr{P}^{\textnormal{!`}}}[W_{\textnormal{Kos}}^{-1}]}(\varphi_{\ast}(S) \otimesC \textnormal{Bar}(A), \textnormal{Bar}(A)) \times_{\map_{\ccoalg_{\mathscr{P}^{\textnormal{!`}}}[W_{\textnormal{Kos}}^{-1}]}(\textnormal{Bar}(A), \textnormal{Bar}(A))} \{\textnormal{id}_{\textnormal{Bar}(A)}\}.
\end{align*}
Consider the unit axiom for the Hopf cooperad $(\mathscr{P}^{\textnormal{!`}})^{\textnormal{cu}}$
\begin{center}
    \begin{tikzcd}
\mathbf{coCom} \otimes_{H}(\mathscr{P}^{\textnormal{!`}})^{\textnormal{cu}} \arrow[rd, "\simeq"'] \arrow[r, "\varphi^{\textnormal{cu}}\otimes \textnormal{id}"] & (\mathscr{P}^{\textnormal{!`}})^{\textnormal{cu}}\otimes_H (\mathscr{P}^{\textnormal{!`}})^{\textnormal{cu}} \arrow[d, "\textnormal{mult}"] \\
                                                                                                                                                          & (\mathscr{P}^{\textnormal{!`}})^{\textnormal{cu}}                                                                                          
\end{tikzcd}.
\end{center}
This shows that $\varphi_{\ast}(S) \otimesC \textnormal{Bar}(A) \simeq \overline{S^+ \otimes \textnormal{Bar}(A)^+}$.
As in the proof of Proposition \ref{prop1}, let $l_A:\ccoalg_{\mathscr{P}^{\textnormal{!`}}}(\ch(\field))_{0/} \rightarrow \ccoalg_{\mathscr{P}^{\textnormal{!`}}}(\ch(\field))_{\text{Bar}(A)/}$ denote the functor
$(0 \rightarrow S) \mapsto (\text{Bar}(A) \rightarrow S\otimesC \text{Bar}(A))$. Let $\tilde{l}_A := l_A \circ \varphi_{\ast}$. Since $\tilde{l}_A$ is a composite of left adjoints, it admits a right adjoint $\tilde{r}_A$. Now let $V\in \ch(\field)$, and consider the trivial conilpotent cocommutative coalgebra structure on $V$. Since taking primitive elements is right adjoint to the trivial conilpotent cocommutative coalgebra functor, we get
\begin{align*}
    \hom_{\ccoalg_{\nccocomm}(\ch(\field))}(V, \tilde{r}_A(\text{id}_{\text{Bar}(A)})) \cong \hom_{\ch(\field)}(V, \text{Prim}(\tilde{r}_A(\text{id}_{\text{Bar}(A)}))).
\end{align*}
The left hand side is isomorphic to 
\begin{align*}
     &\hom_{\ccoalg_{\mathscr{P}^{\text{!`}}}(\ch(\field))_{\text{Bar}(A)/}}(\tilde{l}_A(V), \text{id}_{\text{Bar}(A)})\\&\cong \hom_{\ccoalg_{\mathscr{P}^{\text{!`}}}(\ch(\field))}(\overline{V^+\otimes \text{Bar}(A)^+}, \text{Bar}(A)) \\&\qquad {}\times_{\hom_{\ccoalg_{\mathscr{P}^{\text{!`}}}(\ch(\field))}(\text{Bar}(A),\text{Bar}(A))} \{\text{id}_{\text{Bar}(A)}\} \\&\cong \hom_{\text{CoAlg}^{\text{conil, coaug}}_{(\mathscr{P}^{\text{!`}})^{\text{cu}}}(\ch(\field))}({V^+\otimes \text{Bar}(A)^+}, \text{Bar}(A)^+) \\&\qquad {}\times_{\hom_{\text{CoAlg}^{\text{conil, coaug}}_{(\mathscr{P}^{\text{!`}})^{\text{cu}}}(\ch(\field))}(\text{Bar}(A)^+,\text{Bar}(A)^+)} \{\text{id}_{\text{Bar}(A)}\} \\&\cong \hom_{\ch(\field)}(V, \text{Coder}_{(\mathscr{P}^{\text{!`}})^{\text{cu}}}(\text{Bar}(A)^+)).
\end{align*}
Hence, we get an equivalence of dg Lie algebras $\text{Prim}(\tilde{r}_A(\text{id}_{\text{Bar}(A)})) \simeq \text{Coder}_{(\mathscr{P}^{\text{!`}})^{\text{cu}}}(\text{Bar}(A)^+)$. In particular, we get an equivalence of non-counital cocommutative coalgebras
\begin{align*}
    \tilde{r}_A(\text{id}_{\text{Bar}(A)}) \simeq \overline{U}(\text{Coder}_{(\mathscr{P}^{\text{!`}})^{\text{cu}}}(\text{Bar}(A)^+)).
\end{align*}
The multiplication on the representing coalgebra is induced by composition of the universal action. Under the above identification of primitives, its graded commutator therefore agrees with the commutator of coderivations. Consequently, the above isomorphism is an isomorphism of non-counital cocommutative Hopf algebras.\par
One directly checks that $\text{Coder}_{(\mathscr{P}^{\text{!`}})^{\text{cu}}}(\text{Bar}(A)^+)$ computes the extended operadic tangent complex $\mathbb{T}^{\mathscr{P},\text{ex}}_A$. To see that $\overline{U}(\mathbb{T}^{\mathscr{P}, \text{ex}}_A)$ also represents the functor of $\infty$-categories $(\varphi_{\ast})^{\ast}\underline{\faut}_{\mathscr{P}}(A)$, note that both $l_A$ and $\varphi_{\ast}$ are left Quillen, and therefore so is $\tilde{l}_A$. Call the induced functor on $\infty$-categories $\tilde{L}_A$, and its right adjoint $\tilde{R}_A$. Then by Lemma \ref{lem2}, we get
\begin{align*}
    \map_{(\ccoalg_{\mathscr{P}^{\text{!`}}}(\ch(\field))[W_{\text{Kos}}^{-1}])_{\text{Bar}(A)}/}(\tilde{L}_A(S), \text{id}_{\text{Bar}(A)}) \simeq \map_{\ccoalg_{\nccocomm}(\ch(\field))[W_{\text{Kos}}^{-1}]}(S, \tilde{R}_A(\text{id}_{\text{Bar}(A)})),
\end{align*}
and since $\text{id}_{\text{Bar}(A)}$ is fibrant, $\tilde{R}_A(\text{id}_{\text{Bar}(A)}) \simeq \overline{U}(\mathbb{T}^{\mathscr{P},\text{ex}}_A)$. \par
Finally, by \cite[Chapter 6, Theorem 6.1.2]{GR2}, we have $\overline{U} \simeq C^{\text{CE}}_{\ast} \circ \Omega$. Since $C_{\ast}^{\text{CE}} \simeq \text{Bar}_{\lie\{-1\}} \circ [1]$, Proposition \ref{prop2} shows that the formal group corresponding to the functor $(\varphi_{\ast})^{\ast}\underline{\faut}_{\mathscr{P}}(A)$ corresponds to the $\mathbb{E}_1$-algebra in Lie algebras $\Omega \mathbb{T}^{\mathscr{P},\text{ex}}_A$, or equivalently to the Lie algebra $\mathbb{T}^{\mathscr{P},\text{ex}}_A$.
\end{proof}
\begin{remark}
We call the functor $(\varphi_{\ast})^{\ast} \underline{\faut}_{\mathscr{P}}(A)$ the \textbf{extended formal automorphism group of $A$}. Under connectivity assumptions for $A$ and $\mathscr{P}$, one can think of this as the formal group corresponding to the formal moduli problem of deforming $A$ as a \textbf{curved} $\mathscr{P}_{\infty}$-algebra (or more precisely, of curved $(\mathscr{P}\{1\})_{\infty}$-deformations of $A[-1]$). 
\end{remark}
By the above diagram, we immediately get the following.
\begin{corollary}\label{cor2}
There is an equivalence of non-counital cocommutative Hopf algebras
\begin{align*}
    \varphi^{\textnormal{ext}}(\mathfrak{Z}(\textnormal{Bar}(A))) \simeq \overline{U}(\mathbb{T}^{\mathscr{P},\textnormal{ex}}_A).
\end{align*}
Equivalently, there exists an equivalence of $\mathbb{E}_1$-algebras in Lie algebras
\begin{align*}
    ((\varphi^{\vee})^!)^{\ast}(\mathfrak{Z}(A)[-1]) \simeq \Omega \mathbb{T}^{\mathscr{P}, \textnormal{ex}}_A.
\end{align*}
\end{corollary}
\begin{remark}
The extended operadic tangent complex was defined by Melani-Safronov in \cite{MS}. They show that its Lie algebra structure extends to a natural algebra structure over the brace operad $\textnormal{Br}_{\mathscr{P}^{\textnormal{!`}}}$, and in particular, using the brace additivity constructed in \cite{Saf}, this equips the extended operadic tangent complex with the structure of an $\mathbb{E}_1$-algebra in $\mathscr{P}$-algebras. We expect this structure to agree with the $\mathbb{E}_1$-algebra structure in $\mathscr{P}$-algebras coming from the center of $A$. In fact, Melani-Safronov call this $\textnormal{Br}_{\mathscr{P}^{\textnormal{!`}}}$-algebra the center of $A$.
\end{remark}
\begin{remark}
There is a canonical morphism 
\begin{align*}
    \faut_{\mathscr{P}}(A) \rightarrow (\varphi_{\ast})^{\ast} \underline{\faut}_{\mathscr{P}}(A)
\end{align*}
corresponding to the projection $\overline{S^+ \otimes \text{Bar}(A)^+} \rightarrow \overline{S^+ \otimes \text{Bar}(A)^+}/\varphi_{\ast}(S) \cong S^+ \otimes \text{Bar}(A)$. This yields a fiber sequence of Lie algebras
\begin{align*}
    ((\varphi^{\vee})^!)^{\ast}(A[-1]) \rightarrow \mathbb{T}^{\mathscr{P}}_A \rightarrow \mathbb{T}^{\mathscr{P},\textnormal{ex}}_A.
\end{align*}
We can think of this as the operadic analogue of the fiber sequence in \cite[Theorem 1.1]{FR}.
\end{remark}
\subsection{Examples} \label{the_formal_automorphism_group_3}
Recall that we can apply the theory of formal internal automorphism groups whenever $\mathscr{P}$ is Koszul, $\mathscr{P}^{\text{!`}}$ admits an admissible Hopf counital structure, $\mathscr{P}^!\{-1\}$ is concentrated in non-positive degrees, and $A$ is a $\mathscr{P}$-algebra such that there exists an $n$ such that $\mathscr{P}\{n\}$ and $A[-n]$ are both concentrated in non-positive degrees.\displaypar
This applies to all examples considered in this section. The fact that the Koszul dual cooperads considered here admit admissible Hopf counital structures, and that those correspond to the usual tensor products, was shown in \cite[Corollary 1.18 and Corollary 1.19]{Saf}.\displaypar
First consider $\mathscr{P} = \lie\{-1\}$, so $\mathscr{P}^{\text{!`}} = \nccocomm$ with the Hopf counital structure coming from the counital cocommutative cooperad. Let $\mathfrak{g}$ be a Lie algebra concentrated in non-positive degrees. Then the formal internal automorphism group
\begin{align*}
    \underline{\faut}_{\lie\{-1\}}(\mathfrak{g}[1])|_{\alg^{\text{Art}}_{\nucomm}}
\end{align*}
is a formal group, which corresponds to the $\mathbb{E}_1$-algebra in Lie algebras $\mathfrak{Z}_{\lie\{-1\}}(\mathfrak{g}[1])[-1] \simeq \mathfrak{Z}_{\lie}(\mathfrak{g})$. The map $\varphi$ is the identity on $\nccocomm$, so this formal group is equivalent to the extended formal automorphism group of the shifted Lie algebra $\mathfrak{g}[1]$. In particular, we see that the Lie algebra center is just given by the extended operadic tangent complex
\begin{align*}
\mathfrak{Z}_{\lie}(\mathfrak{g}) \simeq \Omega \mathbb{T}^{\lie\{-1\}, \text{ex}}_{\mathfrak{g}[1]}.
\end{align*}
Now consider $\mathscr{P} = \nuassoc\{-1\}$, so $\mathscr{P}^{\text{!`}} = \nccoassoc$. Again, let $A$ be a non-positive associative algebra. The formal internal automorphism group
\begin{align*}
    \underline{\faut}_{\nuassoc\{-1\}}(A[1])|_{\alg^{\text{Art}}_{\nuassoc}}
\end{align*}
is an associative formal group corresponding to the $\mathbb{E}_1$-algebra in non-unital associative algebras $\mathfrak{Z}_{\nuassoc\{-1\}}(A[1])[-1] \simeq \mathfrak{Z}_{\nuassoc}(A)$. The map $\varphi: \nccocomm \rightarrow \nccoassoc$ corresponds to the canonical map $\varphi^{\vee}: \nuassoc \rightarrow \nucomm$, and the map $(\varphi^{\vee})^!: \lie \rightarrow \nuassoc$ sending an associative algebra to its commutator Lie algebra. By Corollary \ref{cor2}, we get an equivalence of $\mathbb{E}_1$-algebras in Lie algebras
\begin{align*}
    ((\varphi^{\vee})^!)^{\ast} \mathfrak{Z}_{\nuassoc}(A) \simeq \Omega \mathbb{T}^{\nuassoc\{-1\}, \text{ex}}_{A[1]}. 
\end{align*}
We can consider the $\mathbb{E}_1$-algebra in non-unital associative algebras $\mathfrak{Z}(A)$ as a non-unital $\mathbb{E}_2$-algebra using the Dunn additivity map $\rho^{\ast}$ defined in (\ref{eq2}). Let $R_{\mathbb{E}_2}$ denote the map sending an $\mathbb{E}_2$-algebra to its underlying Lie algebra. Then, by Theorem \ref{thm1}, we get an equivalence of Lie algebras
\begin{align*}
    R_{\mathbb{E}_2}((\rho^{\ast})^{-1} \mathfrak{Z}(A)) \simeq \mathbb{T}^{\nuassoc\{-1\}, \text{ex }}_{A[1]}.
\end{align*}
In other words, the Lie algebra structure underlying the $\mathbb{E}_1$-center of $A$ agrees with the Lie algebra structure on the extended operadic tangent complex of $A$, and we recover the morphism of Lie algebras
\begin{align*}
    \mathbb{T}^{\nuassoc\{-1\}}_{A[1]} \rightarrow \mathfrak{Z}_{\nuassoc}(A)[1]
\end{align*}
described in \cite{FR} for the $\mathbb{E}_1$-case.\displaypar
Finally, consider $\mathscr{P} = \mathbb{P}_n\{-n\}$, so $\mathscr{P}^{\text{!`}} = \text{co}\mathbb{P}_n$. Let $A$ be a non-positive $\mathbb{P}_n$-algebra. The formal internal automorphism group 
\begin{align*}
    \underline{\faut}_{\mathbb{P}_n\{-n\}}(A[n])|_{\alg^{\text{Art}}_{\mathbb{P}_n}}
\end{align*}
is a $\mathbb{P}_n$-formal group corresponding to the $\mathbb{E}_1$-algebra in $\mathbb{P}_n\{1-n\}$-algebras $\mathfrak{Z}_{\mathbb{P}_n\{-n\}}(A[n])[-1]\simeq \mathfrak{Z}_{\mathbb{P}_n}(A)[n-1]$. \displaypar
The map $\varphi: \nccocomm \rightarrow \text{co}\mathbb{P}_n$ considers a non-counital cocommutative coalgebra as a non-counital $\text{co}\mathbb{P}_n$-coalgebra with zero cobracket, consequently the dual $\varphi^{\vee}: \mathbb{P}_n \rightarrow \nucomm$ considers a non-unital commutative algebra as a shifted Poisson algebra with zero bracket. The Koszul dual map $(\varphi^{\vee})^!: \lie \rightarrow \mathbb{P}_n\{1-n\}$ forgets the commutative product on the $(n-1)$-shifted version of a $\mathbb{P}_n$-algebra. Again by Corollary \ref{cor2}, we get an equivalence of $\mathbb{E}_1$-algebras in Lie algebras
\begin{align*}
    ((\varphi^{\vee})^!)^{\ast}(\mathfrak{Z}_{\mathbb{P}_n}(A)[n-1]) \simeq \Omega \mathbb{T}^{\mathbb{P}_n\{-n\}, \text{ex}}_{A[n]}. 
\end{align*}
\begin{remark}
By \cite[Proposition 2.21]{Saf}, Poisson additivity $\text{add}: \alg_{\mathbb{P}_{n+1}}(\modu_{\field}) \xrightarrow{\simeq} \alg_{\mathbb{E}_1}(\alg_{\mathbb{P}_n}(\modu_{\field}))$ is compatible with Lie additivity. In particular, consider the $\mathbb{P}_{n+1}$-algebra $\textnormal{add}^{-1}(\mathfrak{Z}_{\mathbb{P}_n}(A))$ corresponding to the $\mathbb{E}_1$-algebra in $\mathbb{P}_n$-algebras $\mathfrak{Z}_{\mathbb{P}_n}(A)$ under Poisson additivity. Then $\textnormal{add}^{-1}(\mathfrak{Z}_{\mathbb{P}_n}(A))[n]$ admits an underlying Lie algebra, which agrees with $\mathbb{T}^{\mathbb{P}_n\{-n\},\textnormal{ex}}_{A[n]}$.
\end{remark}
In the shifted Poisson case, we can actually prove a stronger result, which follows as a corollary of the following result of Tamarkin.
\begin{proposition}[Proposition 3.2 \cite{Tam4}]
Let $A\in \alg_{\mathbb{P}_n}(\ch(\field))$. The functor $\mathcal{G}_A:(\ccoalg_{\textnormal{co}\mathbb{P}_n}(\ch(\field)))^{\textnormal{op}} \rightarrow \mathbf{Set}$,
\begin{align*} 
    S &\mapsto \hom_{\ccoalg_{\textnormal{co}\mathbb{P}_n}(\ch(\field))}(S \otimes_{\textnormal{co}\mathbb{P}_n} \textnormal{Bar}(A[n]), \textnormal{Bar}(A[n])) \\ &\qquad {}\times_{\hom_{\ccoalg_{\textnormal{co}\mathbb{P}_n}(\ch(\field))}(\textnormal{Bar}(A[n]), \textnormal{Bar}(A[n]))} \{\textnormal{id}_{\textnormal{Bar}(A[n])}\}
\end{align*}
is represented by $\textnormal{Bar}(\map_{\field}(\textnormal{Bar}(A[n])^+,A)[n])$, where $\map_{\field}(\textnormal{Bar}(A[n])^+,A)$ is equipped with the convolution $\mathbb{P}_n$-algebra structure coming from $\textnormal{Bar}(A[n])^+$ being a (counital) $\textnormal{co}\mathbb{P}_n$-coalgebra and $A$ a $\mathbb{P}_n$-algebra, and with differential twisted by the Maurer--Cartan element corresponding to $\textnormal{id}_{\textnormal{Bar}(A[n])}$.
\end{proposition}
\begin{corollary}
We have an equivalence of $\textnormal{co}\mathbb{P}_n$-coalgebras
\begin{align*}
    \mathfrak{Z}(\textnormal{id}_{\textnormal{Bar}(A[n])}) \simeq \textnormal{Bar}(\map_{\field}(\textnormal{Bar}(A[n])^+,A)[n])
\end{align*}
\end{corollary}
\begin{proof}
As in the proof of Proposition \ref{prop1}, let 
\[
l_{A[n]}:\ccoalg_{\text{co}\mathbb{P}_n}(\ch(\field))_{0/} \rightarrow \ccoalg_{\text{co}\mathbb{P}_n}(\ch(\field))_{\text{Bar}(A[n])/}
\]
 denote the functor
$(0 \rightarrow S) \mapsto (\text{Bar}(A[n]) \rightarrow S\otimes_{\text{co}\mathbb{P}_n} \text{Bar}(A[n]))$. The functor $\mathcal{G}_A$ is given by $S\mapsto \hom_{\ccoalg_{\text{co}\mathbb{P}_n}(\ch(\field))_{\text{Bar}(A[n])/}}({l}_{A[n]}(S), \text{id}_{\text{Bar}(A[n])})$. By the above proposition, the right adjoint $r_{A[n]}$ satisfies $r_{A[n]}(\text{id}_{\text{Bar}(A[n])}) \simeq \textnormal{Bar}(\map_{\field}(\textnormal{Bar}(A[n])^+,A)[n])$, and since $\text{Bar}(A[n])$ is fibrant, the same holds for the induced right adjoint on $\infty$-categories $R_{A[n]}$. In particular, $\textnormal{Bar}(\map_{\field}(\textnormal{Bar}(A[n])^+,A)[n])$ represents the formal internal automorphism group of $A[n]$, and is therefore equivalent to the centralizer of $\text{id}_{\text{Bar}(A[n])}$.
\end{proof}
By giving the representing object $\textnormal{Bar}(\map_{\field}(\textnormal{Bar}(A[n])^+,A)[n])$ the $\mathbb{E}_1$-algebra structure induced by composition, this equivalence extends to an equivalence of $\mathbb{E}_1$-algebras in $\text{co}\mathbb{P}_n$-coalgebras
\begin{align*}
    \mathfrak{Z}(\text{Bar}(A[n])) \simeq \textnormal{Bar}(\map_{\field}(\textnormal{Bar}(A[n])^+,A)[n]).
\end{align*}
In \cite[Definition 3.7]{MS}, V. Melani and P. Safronov define a center of a homotopy algebra over a Koszul operad whose Koszul dual cooperad admits a Hopf counital structure. For the shifted Poisson case, the Melani-Safronov center of $A$ is given by 
\begin{align*}
    Z_{\text{MS}}(A) := \text{Conv}(\text{co}\mathbb{P}_n^{\text{cu}}, \e_{A[n]})^{\phi}[-n] \cong \mathbb{T}^{\mathbb{P}_n\{-n\},\text{ex}}_{A[n]}[-n].
\end{align*}
This is canonically an algebra over the operad $\text{Br}_{\text{co}\mathbb{P}_n}\{n\}$, with $\text{Br}_{\text{co}\mathbb{P}_n}$ the $\text{co}\mathbb{P}_n$-braces operad constructed in \cite{CW}. By \cite[Theorem 4]{CW}, we have a zig-zag of equivalences of operads
\begin{align*}
    \mathbb{P}_{n+1} \xleftarrow{\simeq} \Omega (\text{co}\mathbb{P}_{n+1}\{n+1\}) \xrightarrow{\simeq} \textnormal{Br}_{\textnormal{co}\mathbb{P}_n}\{n\}.  
\end{align*}
\begin{theorem}\label{thm3}
Let $A \in \alg_{\mathbb{P}_n}(\ch(\field))$. Under Poisson additivity, the $\mathbb{E}_1$-algebra in $\mathbb{P}_n$-algebras given by the center $\mathfrak{Z}(A)$ corresponds to the homotopy $\mathbb{P}_{n+1}$-algebra structure on the Melani-Safronov center $Z_{\textnormal{MS}}(A)$ given by its canonical $\textnormal{Br}_{\textnormal{co}\mathbb{P}_n}\{n\}$-algebra structure. 
\end{theorem}
\begin{lemma}\label{lem3}
Let $\mathscr{C}$ be a dg cooperad with $\mathscr{C}(0) = 0$ and $\mathscr{C}(1) = \field$, equipped with Hopf counital structure $\mathscr{C}^{\textnormal{cu}}$, and $B \in \ch(\field)$. Equip $\textnormal{Conv}(\mathscr{C}^{\textnormal{cu}},\e_B)$ with the canonical $\textnormal{preLie}_{\mathscr{C}}$-algebra structure defined in \cite{CW}. Denote by $p: \mathscr{C}^{\textnormal{cu}}(B) \rightarrow B$ and $\pi: \mathscr{C}^{\textnormal{cu}}(\textnormal{Conv}(\mathscr{C}^{\textnormal{cu}},\e_B)) \rightarrow \textnormal{Conv}(\mathscr{C}^{\textnormal{cu}},\e_B)$ the respective projections onto cogenerators. Let $e_1: \mathscr{C}^{\textnormal{cu}}(\textnormal{Conv}(\mathscr{C}^{\textnormal{cu}}, \e_B)) \otimes_{\mathscr{C}^{\text{cu}}} \mathscr{C}^{\textnormal{cu}}(B) \rightarrow B$ be given by $e_1(u,x) = \epsilon(u)p(x) + \pi(u)(x)$, and let $e$ be its coextension to $\mathscr{C}^{\textnormal{cu}}(B)$. Then for $f,g_1,\dots,g_r\in \textnormal{Conv}(\mathscr{C}^{\textnormal{cu}},\e_B) \cong \map_{\field}(\mathscr{C}^{\textnormal{cu}}(B),B)$, $x\in \mathscr{C}^{\textnormal{cu}}(B)$, and $c\in \mathscr{C}^\textnormal{cu}(r)$, 
\begin{align*}
    f(e((c;g_1,\dots,g_r),x)) = (f\{c|g_1,\dots,g_r\})(x).
\end{align*}
\end{lemma}
\begin{proof}
Assume $x= (d;b_1,\dots,b_k)$. Consider the $n$th component of the coaction 
\begin{multline*}
\Delta_n: \mathscr{C}^{\text{cu}}(\map_{\field}(\mathscr{C}^{\text{cu}}(B),B)) \otimes_{\mathscr{C}^{\text{cu}}} \mathscr{C}^{\text{cu}}(B) \\
\rightarrow \mathscr{C}^{\text{cu}}(n) \otimes_{\Sigma_n} (\mathscr{C}^{\text{cu}}(\map_{\field}(\mathscr{C}^{\text{cu}}(B),B)) \otimes_{\mathscr{C}^{\text{cu}}} \mathscr{C}^{\text{cu}}(B))^{\otimes n}.
\end{multline*}
This is given by 
\begin{align*}
    \Delta_n((c;g_1,\dots,g_r),(d;b_1,\dots,b_k)) = \frac{1}{n!}\sum_{I,J} \sum_{\Delta_I(c),\Delta_J(d)} \pm c_0d_0 \otimes_{\Sigma_n} \bigotimes_{i=1}^n \left((c_i;g_{I_i})\otimes (d_i;b_{J_i})\right)
\end{align*}
where the first sum is over partitions $\{1,\dots,r\} = I_1 \coprod \dots \coprod I_n$ and $\{1,\dots,k\} = J_1\coprod \dots \coprod J_n$, and the second sum encompasses the sums in $\Delta_I(c) = \sum c_0 \otimes c_1 \otimes \dots \otimes c_n$ and $\Delta_J(d) = \sum d_0 \otimes d_1 \otimes \dots \otimes d_n$. Then the $n$th component of the map $e$ is given by $e_n = \left( \text{id}_{\ccu} \otimes_{\Sigma_n} e_1^{\otimes n} \right)\circ \Delta_n$. We have
\begin{align*}
    e_1((c_i;g_{I_i}), (d_i;b_{J_i})) = \begin{cases}
        c_i g_a(d_i;b_{J_i}) & \text{if } I_i = \{a\}\\
        c_i d_i b_j & \text{if } I_i = \emptyset \text{ and } J_i = \{j\}\\
        0 & \text{otherwise}
    \end{cases}.
\end{align*}
Hence, the only contributions to $e_n((c;g_1,\dots,g_r),(d;b_1,\dots,b_k))$ come from summands with $I_i$ either the empty set or a singleton for each $i$. Such a partition is uniquely defined by an injection $\{1,\dots,r\} \hookrightarrow \{1,\dots,n\}$, which we interpret as specifying the singleton cases. In this case, $\Delta_I: \ccu(r) \rightarrow \ccu(n)$ is just the extension map coming from the nullary part of $\ccu$. For the partition $J$, it then suffices to specify the $j \in \{1,\dots,k\}$ for which there is an $i$ with $I_i = \emptyset$ and $J_i = \{j\}$, and the sets $J_i$ in case $I_i$ is a singleton. We can hence sum instead over partitions $p: \{1,\dots,k\} = J_0 \coprod J_1 \coprod \dots \coprod J_r$. In this notation, $\Delta_J$ is the decomposition along the tree $T_p$
\begin{center}
$T_p = \quad$
\begin{tikzpicture}[
  baseline=-0.5ex,
  x=1cm, y=1cm,
  line width=0.55pt,
  line cap=round,
  vertex/.style={
    circle, draw, fill=white,
    minimum size=5mm, inner sep=0pt
  },
  every node/.style={font=\normalsize}
]
  \node[vertex] (v0) at (0,0)       {$0$};
  \node[vertex] (v1) at (-1.65,-.6) {$1$};
  \node[vertex] (vr) at (-.85,1.45) {$r$};

  \draw (v0) -- (v1);
  \draw (v0) -- (vr);
  \draw (v0) -- (0,-1.15);

  \draw (v1) -- (-1.65,-1.6)
    node[below=3pt] {$j_1^1$};
  \draw (v1) -- (-2.75,-.2)
    node[left=3pt] {$j_1^{|J_1|}$};
  \node[rotate=-65] at (-2.25,-1.05) {$\cdots$};

  \node[rotate=65] at (-1.15,.4) {$\cdots$};

  \draw (vr) -- (-1.85,.85)
    node[below left=2pt] {$j_r^1$};
  \draw (vr) -- (-.4,2.3)
    node[above right=2pt] {$j_r^{|J_r|}$};
  \node[rotate=45] at (-1.5,1.95) {$\cdots$};

  \draw (v0) -- (.8,1.2)
    node[above right=2pt] {$j_0^1$};
  \draw (v0) -- (1.2,-.35)
    node[right=3pt] {$j_0^{|J_0|}$};
  \node at (1,.5) {$\vdots$};
\end{tikzpicture}.
\end{center}
In particular, we get
\begin{multline*}
    f(e((c;g_1,\dots,g_r),(d;b_1,\dots,b_k))) \\
= \sum_{p: \{1,\dots,k\} = J_0 \coprod J_1\coprod \dots \coprod J_r} \sum_{\Delta_{T_p}(d)} \pm f_{r + |J_0|}(c d_0; g_1(d_1;b_{J_1}),\dots, g_r(d_r;b_{J_r}), b_{J_0}).
\end{multline*}
By definition, this is precisely $(f\{c|g_1,\dots,g_r\})(d;b_1,\dots,b_k)$.
\end{proof}
\begin{proof}[Proof of Theorem \ref{thm3}]
The operad map $\Omega \text{co}\mathbb{P}_n \rightarrow \text{Br}_{\text{co}\mathbb{P}_n}$ described in \cite[Section 3.1]{MS} induces a homotopy $\mathbb{P}_n\{-n\}$ algebra structure on $Z_{\text{MS}}(A)[n]$. By \cite[Proposition 2.17]{Saf}, it suffices to show that the $\mathbb{E}_1$-algebra structure in $\text{co}\mathbb{P}_n$-coalgebras on Tamarkin's $\text{Bar}(\map_{\field}(\text{Bar}(A[n])^+, A[n]))$ agrees with the $\mathbb{E}_1$-algebra structure in $\text{co}\mathbb{P}_n$-coalgebras induced on $\text{Bar}(Z_{\text{MS}}(A)[n])$ by brace additivity \cite[Proposition 2.8]{Saf}.\par
First, compare the underlying $\text{co}\mathbb{P}_n$-coalgebra structures. To this end, it suffices to show that the $\Omega \text{co}\mathbb{P}_n$-algebra structure on $Z_{\text{MS}}(A)[n]$ induced by the brace algebra structure agrees with the $\mathbb{P}_n\{-n\}$-algebra structure on $\map_{\field}(\text{Bar}(A[n])^+, A[n])$ given by convolution. Let $m \in \mathbb{P}_n(2)$ denote the generator of the product, and $\lambda \in \mathbb{P}_n(2)$ the generator of the bracket. Let $u = m^{\vee}$ and $v = \lambda^{\vee}$ denote the dual generators in $\text{co}\mathbb{P}_n(2)$. Since the $\mathbb{P}_n$-algebra structure on $A$ is strict, the Maurer-Cartan element $\phi$ is supported only in arity 2. The component $\phi_2: \text{co}\mathbb{P}_n(2) \rightarrow \e_{A[n]}$ sends $u$ to the bracket $\{-,-\}$ of $A$ and $v$ to the product $\cdot$ of $A$. The Hopf comultiplication on $\mathbb{P}_n^{\text{cu}}$ is given by $\Delta(m) = m\otimes m$ and $\Delta(\lambda) = \lambda \otimes m + m\otimes \lambda$. Therefore, the Hopf multiplication on $\text{co}\mathbb{P}_n$ satisfies $u^2 = u$, $v^2 = 0$, $uv =v$. Thus, for the action of $u$ on $Z_{\text{MS}}(A)[n]$, we get 
\begin{align*}
    m(u|f_p,g_q)(d;a_1,\dots,a_{p+q}) = \sum_{I\coprod J} \sum_{\Delta_{I,J}(d)} \pm \{f_p(d_I^u;a_I), g_q(d_J^u;a_J)\} + f_p(d_I^v;a_I)\cdot g_q(d_J^v;a_J),
\end{align*}
and for the action of $v$ we get
\begin{align*}
    m(v|f_p,g_q)(d;a_1,\dots,a_{p+q}) = \sum_{I\coprod J} \sum_{\Delta_{I,J}(d)} \pm f_p(d_I^u;a_I) \cdot g_q(d_J^u;a_J)
\end{align*}
with $\Delta_{I,J}(d) = \sum u \otimes d_I^u \otimes d_J^u + \sum v \otimes d_I^v \otimes d_J^v$. These are precisely the convolution operations. \par
Applying brace additivity, the $\mathbb{E}_1$-algebra structure on $\text{Bar}(Z_{\text{MS}}(A)[n])$ is given by the braces
\begin{align*}
    \text{co}\mathbb{P}_n^{\text{cu}}(r) \otimes Z_{\text{MS}}(A)[n] \otimes (Z_{\text{MS}}(A)[n])^{\otimes r} &\rightarrow Z_{\text{MS}}(A)[n]\\
    (c, f, g_1,\dots, g_r) &\mapsto f\{c|g_1,\dots,g_r\}.
\end{align*}
By the proof of \cite[Proposition 3.2]{Tam4}, we see that the universal action map 
\begin{align*}
\text{Bar}(\map_{\field}(\text{Bar}(A[n])^+,A[n]))^+ \otimes_{\text{co}\mathbb{P}_n^{\text{cu}}} \text{Bar}(A[n])^+ \rightarrow \text{Bar}(A[n])^+
\end{align*}
corresponds precisely to the map $e$ of Lemma \ref{lem3} with $\mathscr{C} = \text{co}\mathbb{P}_n$ and $B= A[n]$. The $\mathbb{E}_1$-multiplication of Tamarkin on $\text{Bar}(\map_{\field}(\text{Bar}(A[n])^+,A[n]))^+$ is defined using this universal action as corresponding to 
\begin{align*}
    &(\text{Bar}(\map_{\field}(\text{Bar}(A[n])^+,A[n]))^+ \otimes_{\text{co}\mathbb{P}_n^{\text{cu}}} \text{Bar}(\map_{\field}(\text{Bar}(A[n])^+,A[n]))^+) \otimes_{\text{co}\mathbb{P}_n^{\text{cu}}} \text{Bar}(A[n])^+ \\&\xrightarrow{\text{id}\otimes e} \text{Bar}(\map_{\field}(\text{Bar}(A[n])^+,A[n]))^+ \otimes_{\text{co}\mathbb{P}_n^{\text{cu}}} \text{Bar}(A[n])^+ \\&\xrightarrow{e} \text{Bar}(A[n])^+.
\end{align*}
Hence, by Lemma \ref{lem3}, this agrees with the multiplication from braces.
\end{proof}
\begin{corollary}\label{cor3}
Let $A$ be a cofibrant commutative algebra concentrated in non-positive degrees. Then the classical $\mathbb{P}_2$-algebra structure on polyvector fields $T_{\textnormal{poly}}(A) := \hom_A(\sym_A \Omega_A^1[1],A)$ given by the Schouten bracket is equivalent to the homotopy $\mathbb{P}_2$-algebra $\textnormal{add}^{-1}(\mathfrak{Z}_{\mathbb{P}_1}(A))$, where we view $A$ as a Poisson algebra with zero bracket.
\end{corollary}
\begin{proof}
    This follows directly from the above theorem and the fact that polyvector fields agree with the Melani-Safronov center for cofibrant algebras.
\end{proof}
\begin{remark}
If $A$ is a smooth finitely presented classical commutative $\field$-algebra, we can make the same conclusion even without the cofibrancy hypothesis. Indeed, the Harrison complex is then quasi-isomorphic to the projective module of Kähler differentials, so the same comparison argument applies.
\end{remark}

%% file: appendix.tex
Let $\mathcal{C}^{\otimes}$ be a symmetric monoidal $\infty$-category. Then Lurie's version of the Dunn additivity theorem states \cite[Theorem 5.1.2.2]{HA} that there is an equivalence of $\infty$-categories
\begin{align}\label{eq2}
    \rho^{\ast}: \alg_{\mathbb{E}_{k+k'}}(\mathcal{C}) \xrightarrow{\simeq} \alg_{\mathbb{E}_k}(\alg_{\mathbb{E}_{k'}}(\mathcal{C})).
\end{align}
Here $\rho: \mathbb{E}_k^{\otimes} \times \mathbb{E}_{k'}^{\otimes} \rightarrow \mathbb{E}_{k+k'}^{\otimes}$ is the bifunctor of $\infty$-operads which is given explicitly on topological operads as Cartesian products of little squares (see \cite[Construction 5.1.2.1]{HA}).\displaypar
On the other hand, there exists a Lie additivity functor 
\begin{align*}
    \Omega:  \alg_{\lie}(\modu_{\field}) \xrightarrow{\simeq} \alg_{\mathbb{E}_1}(\alg_{\lie}(\modu_{\field}))
\end{align*}
given by taking loops in the pointed Cartesian symmetric monoidal $\infty$-category of Lie algebras. In \cite{Saf}, Safronov identifies this with the brace additivity functor for the Lie operad. \displaypar
There is a forgetful functor $R_{\mathbb{E}_n}: \alg_{\mathbb{E}_n}(\modu_{\field}) \rightarrow \alg_{\lie}(\modu_{\field})$ sending an $\mathbb{E}_n$-algebra $A$ to the underlying Lie bracket on $A[n-1]$. This admits a left adjoint 
\begin{align*}
    U_{\mathbb{E}_n}: \alg_{\lie}(\modu_{\field}) \rightarrow \alg_{\mathbb{E}_n}(\modu_{\field}),
\end{align*}
which is the $\mathbb{E}_n$-universal enveloping algebra functor. This upgrades to a functor 
\begin{align*}
    U^{\text{Hopf}}_{\mathbb{E}_n}: \alg_{\lie}(\modu_{\field}) \rightarrow \alg_{\mathbb{E}_n}(\coalg_{\textbf{coCom}}(\modu_{\field})).
\end{align*}
\begin{theorem}[\cite{Kn}, \cite{ACBH}, \cite{GR2}]\label{thm2}
    The functor $U^{\textnormal{Hopf}}_{\mathbb{E}_n}$ is equivalent to $\alg_{\mathbb{E}_n}(C^{\textnormal{CE}}_{\ast}) \circ \Omega^{ n}$, with $C_{\ast}^{\textnormal{CE}}: \alg_{\lie}(\modu_{\field}) \rightarrow \coalg_{\textnormal{\textbf{coCom}}}(\modu_{\field})$ the Chevalley-Eilenberg chains functor.
\end{theorem}
In particular, the functor $U_{\mathbb{E}_n}$ is symmetric monoidal, and therefore its right adjoint is lax symmetric monoidal. The goal of this appendix is to prove the following compatibility
\begin{theorem}\label{thm1}
The following diagram commutes
\begin{center}
    \begin{tikzcd}
\alg_{\mathbb{E}_2}(\modu_{\field}) \arrow[r, "\rho^{\ast}"] \arrow[d, "R_{\mathbb{E}_2}"'] & \alg_{\mathbb{E}_1}(\alg_{\mathbb{E}_1}(\modu_{\field})) \arrow[d, "\alg_{\mathbb{E}_1}(R_{\mathbb{E}_1})"] \\
\alg_{\lie}(\modu_{\field}) \arrow[r, "\Omega"']                                            & \alg_{\mathbb{E}_1}(\alg_{\lie}(\modu_{\field}))                                                           
\end{tikzcd}
\end{center}
\end{theorem}
To this end, we will prove that the corresponding diagram for the left adjoints commutes, i.e.
\begin{center}
    \begin{tikzcd}
\alg_{\mathbb{E}_2}(\modu_{\field}) \arrow[r, "\rho^{\ast}"]                   & \alg_{\mathbb{E}_1}(\alg_{\mathbb{E}_1}(\modu_{\field}))                                             \\
\alg_{\lie}(\modu_{\field}) \arrow[r, "\Omega"'] \arrow[u, "U_{\mathbb{E}_2}"] & \alg_{\mathbb{E}_1}(\alg_{\lie}(\modu_{\field})) \arrow[u, "\alg_{\mathbb{E}_1}(U_{\mathbb{E}_1})"']
\end{tikzcd}.
\end{center}
\begin{lemma}
Let $\mathcal{C}$ be a Cartesian symmetric monoidal $\infty$-category with finite limits. The following diagram commutes
\begin{center}
\begin{tikzcd}
\mathcal{C}_{\ast} \arrow[r, "\Omega^k"] \arrow[d, "\Omega^{k+k'}"'] & \alg_{\mathbb{E}_k}(\mathcal{C}) \arrow[d, "\alg_{\mathbb{E}_k}(\Omega^{k'})"] \\
\alg_{\mathbb{E}_{k+k'}}(\mathcal{C}) \arrow[r, "\rho^{\ast}"']      & \alg_{\mathbb{E}_k}(\alg_{\mathbb{E}_{k'}}(\mathcal{C}))                     
\end{tikzcd}
\end{center}
\end{lemma}
\begin{proof}
Let $X\in \mathcal{C}_{\ast}$ be a pointed object, and consider the two functors
\begin{align*}
    \mathbb{E}_k^{\otimes} \times \mathbb{E}_{k'}^{\otimes} &\xrightarrow{\rho} \mathbb{E}_{k+k'}^{\otimes} \xrightarrow{\Omega^{k+k'}(X)} \mathcal{C},\\
    \mathbb{E}_k^{\otimes} \times \mathbb{E}_{k'}^{\otimes} &\xrightarrow{(\alg_{\mathbb{E}_k}(\Omega^{k'})\circ \Omega^k)(X)} \mathcal{C}. 
\end{align*}
We need to show that these two functors agree. To this end, consider the $\mathbb{E}_r$-action on $\Omega^r X$ as in \cite[Remark 5.2.6.25]{HA}. We have 
\begin{align*}
    \Omega^r(X) \simeq \map_{\ast}(C_r^+, X)
\end{align*}
with $C_r = (0,1)^r$ the $r$-dimensional open square, and $C_r^+$ its one-point compactification. If $c: \coprod_{i=1}^n C_r \rightarrow C_r$ is an arity $n$ operation in $\mathbb{E}_r$, we have a corresponding collapse map $\delta_c: C_r^+ \rightarrow \bigvee_{i=1}^n C_r^+$. Then $c$ acts on $\Omega^r X$ as
\begin{align*}
    (\Omega^r X)^{\times n} &\simeq \map_{\ast}(C_r^+, X)^{\times n} \\&\simeq \map_{\ast}(\bigvee_{i=1}^n C_r^+, X) \\&\xrightarrow{\delta_c^{\ast}} \map_{\ast}(C_r^+,X) \simeq \Omega^r X.
\end{align*}
If $c\in \mathbb{E}_k(m)$ and $d\in \mathbb{E}_{k'}(n)$, then we get $\rho(c,d)\in \mathbb{E}_{k+k'}(mn)$ with collapse map
\begin{align*}
    \delta_{\rho(c,d)} \simeq \delta_c \wedge \delta_d: C_{k+k'}^+ \simeq (C_k \times C_{k'})^+ \simeq C_k^+ \wedge C_{k'}^+ \xrightarrow{\delta_c \wedge \delta_d} \bigvee_{i=1}^m \bigvee_{j=1}^n C_k^+ \wedge C_{k'}^+ \simeq \bigvee_{i=1}^{mn} C_{k+k'}^+.
\end{align*}
It hence suffices to show that on $(\alg_{\mathbb{E}_k}(\Omega^{k'}) \circ \Omega^k)(X)$, the pair $(c,d)$ acts as precomposition with $\delta_c\wedge \delta_d$. But this follows directly from the tensor-power adjunction on $\map_{\ast}(C_k^+, \map_{\ast}(C_{k'}^+,X))$ and the description of the $\mathbb{E}_r$-algebra structure on $\Omega^r X$.
\end{proof}
\begin{corollary}
Let $\mathfrak{g}$ be a Lie algebra. Then $\rho^{\ast}(\Omega^2 \mathfrak{g}) \simeq \alg_{\mathbb{E}_1}(\Omega)(\Omega \mathfrak{g})$.
\end{corollary}
\begin{proof}[Proof of Theorem \ref{thm1}]
By definition, we have 
\begin{align*}
    \rho^{\ast}_{\textbf{coCom}} \circ \alg_{\mathbb{E}_2}(C_{\ast}^{\text{CE}}) \simeq \alg_{\mathbb{E}_1}(\alg_{\mathbb{E}_1}(C_{\ast}^{\text{CE}})) \circ \rho^{\ast}_{\lie}.
\end{align*}
Applying $\alg_{\mathbb{E}_1}(\alg_{\mathbb{E}_1}(C_{\ast}^{\text{CE}}))$ to the above corollary, we hence get
\begin{align*}
    \rho^{\ast}_{\textbf{coCom}}\circ \alg_{\mathbb{E}_2}(C_{\ast}^{\text{CE}}) \circ \Omega^2 \simeq \alg_{\mathbb{E}_1}(\alg_{\mathbb{E}_1}(C_{\ast}^{\text{CE}})) \circ \rho^{\ast}_{\lie} \circ \Omega^2\simeq \alg_{\mathbb{E}_1}(\alg_{\mathbb{E}_1}(C_{\ast}^{\text{CE}}) \circ \Omega) \circ \Omega.
\end{align*}
By Theorem \ref{thm2}, we can identify $\alg_{\mathbb{E}_2}(C_{\ast}^{\text{CE}}) \circ \Omega^2 \simeq U_{\mathbb{E}_2}^{\text{Hopf}}$ and $\alg_{\mathbb{E}_1}(C_{\ast}^{\text{CE}}) \circ \Omega \simeq U_{\mathbb{E}_1}^{\text{Hopf}}$. Thus
\begin{align*}
    \rho^{\ast}_{\textbf{coCom}}\circ U_{\mathbb{E}_2}^{\text{Hopf}} \simeq \alg_{\mathbb{E}_1}(U^{\text{Hopf}}_{\mathbb{E}_1}) \circ \Omega.
\end{align*}
\end{proof}